\documentclass[a4paper, 11pt]{article}

\usepackage[utf8]{inputenc}
\usepackage[T1]{fontenc}

\usepackage{amsthm}
\usepackage{amssymb}
\usepackage{mathtools}
\usepackage{stmaryrd}	\SetSymbolFont{stmry}{bold}{U}{stmry}{m}{n}	% To avoid warnings
\usepackage{bm}
\usepackage{mathrsfs}
\usepackage{euscript}
\usepackage{tikz-cd}
\usetikzlibrary{positioning, shapes.geometric, patterns, graphs, decorations.pathmorphing}

\usepackage{paralist}

\usepackage[colorlinks]{hyperref}
\hypersetup{
	linkcolor=blue,
	citecolor=red,
	urlcolor=teal,
	pdftitle = {},
	pdfauthor = {Wen-Wei Li}
}

\newcommand{\Z}{\ensuremath{\mathbb{Z}}}

\newcommand{\R}{\ensuremath{\mathbb{R}}}
\newcommand{\CC}{\ensuremath{\mathbb{C}}}

\newcommand{\gr}{\operatorname{gr}}
\newcommand{\Aut}{\operatorname{Aut}}
\newcommand{\Tr}{\operatorname{tr}}
\newcommand{\Sym}{\operatorname{Sym}}

\newcommand{\dd}{\mathop{}\!\mathrm{d}}
\newcommand{\lrangle}[1]{\ensuremath{\langle #1 \rangle}}

\newcommand{\Stab}{\ensuremath{\mathrm{Stab}}}

\newcommand{\identity}{\ensuremath{\mathrm{id}}}

\newcommand{\Hom}{\operatorname{Hom}}
\newcommand{\iHom}{\ensuremath{\mathrm{R}\mathcal{H}{om}}}

\newcommand{\End}{\operatorname{End}}
\newcommand{\rightiso}{\ensuremath{\stackrel{\sim}{\rightarrow}}}
\newcommand{\leftiso}{\ensuremath{\stackrel{\sim}{\leftarrow}}}
\newcommand{\Rder}{\operatorname{R}\!}	% Right derived functor
\newcommand{\Lder}{\operatorname{L}\!}	% Right derived functor

\newcommand{\Image}{\operatorname{im}}
\newcommand{\dotimes}[1]{\ensuremath{\underset{#1}{\otimes}}}

\newcommand{\Lie}{\operatorname{Lie}}
\newcommand{\Ad}{\operatorname{Ad}}
\newcommand{\Spec}{\operatorname{Spec}}
\newcommand{\Gm}{\ensuremath{\mathbb{G}_\mathrm{m}}}
\newcommand{\Ga}{\ensuremath{\mathbb{G}_\mathrm{a}}}

\newcommand{\utimes}[1]{\ensuremath{\overset{#1}{\times}}}
\newcommand{\dtimes}[1]{\ensuremath{\underset{#1}{\times}}}
\newcommand{\Supp}{\operatorname{Supp}}
\newcommand{\ana}{\ensuremath{\mathrm{an}}}	% Analytification

\newcommand{\GL}{\operatorname{GL}}

\theoremstyle{plain}
\newtheorem{proposition}{Proposition}
\newtheorem{lemma}[proposition]{Lemma}
\newtheorem{theorem}[proposition]{Theorem}
\newtheorem{corollary}[proposition]{Corollary}
\theoremstyle{definition}
\newtheorem{definition}[proposition]{Definition}
\newtheorem{definition-theorem}[proposition]{Definition--Theorem}
\newtheorem{definition-proposition}[proposition]{Definition--Proposition}
\newtheorem{remark}[proposition]{Remark}

\newtheorem{example}[proposition]{Example}

\theoremstyle{definition}

\theoremstyle{plain}

\numberwithin{equation}{section}
\numberwithin{proposition}{section}
\numberwithin{conj}{section}	% In the Introduction only

\newcommand{\cate}[1]{\ensuremath{\mathsf{#1}}}	% Font series for categories
\newcommand{\dcate}[1]{\ensuremath{\text{-}\mathsf{#1}}}	% Categories with a dash on the left
\newcommand{\Loc}{\ensuremath{\mathrm{Loc}}}		% Localization functor

\renewcommand{\emptyset}{\ensuremath{\varnothing}}	% Symbol for the emptyset

\setdefaultitem{$\star$}{\normalfont\bfseries \textendash}{\textasteriskcentered}{\textperiodcentered}	% Set the itemize environment

\usepackage{amsmath}
\usepackage[nottoc]{tocbibind}

\usepackage{geometry}
\usepackage{lmodern}

\title{On the regularity of $D$-modules generated by relative characters}
\author{Wen-Wei Li}
\date{}
\makeatletter
\renewcommand{\l@section}{\@dottedtocline{1}{1.5em}{2.0em}}
\renewcommand{\l@subsection}{\@dottedtocline{2}{4.0em}{3.0em}}
\makeatother

\begin{document}

\maketitle

\begin{abstract}
	Following the ideas of Ginzburg, for a subgroup $K$ of a connected reductive $\mathbb{R}$-group $G$ we introduce the notion of $K$-admissible $D$-modules on a homogeneous $G$-variety $Z$. We show that $K$-admissible $D$-modules are regular holonomic when $K$ and $Z$ are absolutely spherical. This framework includes: (i) the relative characters attached to two spherical subgroups $H_1$ and $H_2$, provided that the twisting character $\chi_i$ factors through the maximal reductive quotient of $H_i$, for $i = 1, 2$; (ii) localization on $Z$ of Harish-Chandra modules; (iii) the generalized matrix coefficients when $K(\mathbb{R})$ is maximal compact. This complements the holonomicity proven by Aizenbud--Gourevitch--Minchenko. The use of regularity is illustrated by a crude estimate on the growth of $K$-admissible distributions which based on tools from subanalytic geometry.
\end{abstract}

%{\scriptsize
%\begin{tabular}{ll}
%	\textbf{MSC (2010)} & 11F70; 11R58 22E55 \\
%	\textbf{Keywords} &
%\end{tabular}}

\tableofcontents

\section{Introduction}\label{sec:intro}

Let $G$ be a connected reductive group over $\R$ and denote its opposite group by $G^{\mathrm{op}}$. Differential equations with regular singularities have played an important role in representation theory of the Lie group $G(\R)$. One significant example is Harish-Chandra's study of invariant eigendistributions on $G(\R)$, which includes the character
\[ f \mapsto \Theta_\pi(f) := \Tr \pi(f), \quad f \in C^\infty_c(G(\R)) \]
of an SAF representation $\pi$ of $G(\R)$ as a typical case. Our terminology of SAF representation follows \cite{BK14}, meaning smooth admissible Fréchet of moderate growth. Another example is the study of asymptotics of the matrix coefficients
\[ \lrangle{\check{v}, \pi(\cdot) v} \in C^\infty(G(\R)), \quad v \in V_\pi, \; \check{v} \in V_{\check{\pi}} \]
of these representations, as exemplified by \cite{CM82}; here $V_\pi$ stands for the underlying Fréchet space, $\check{\pi}$ for the contragredient representation, and $\lrangle{\cdot, \cdot}$ for the canonical pairing.

Generalizing the characters or matrix coefficients to the relative setting, one can also consider similar distributions on $Z(\R)$, where $Z$ is an $\R$-variety with $Z(\R) \neq \emptyset$, homogeneous under right $G$-action, satisfying finiteness condition under some subgroup $K \subset G$ and the center $\mathcal{Z}(\mathfrak{g})$ of $U(\mathfrak{g})$. Of course, $Z$ and $K$ must be subject to some geometric conditions. It turns out that \emph{sphericity} is a reasonable requirement. In this article, we say $Z$ is spherical if $Z_{\CC} := Z \times_{\R} \CC$ has an open dense orbit under any Borel subgroup of $G_{\CC} := G \times_{\R} \CC$, and we say $K \subset G$ is spherical if the homogeneous variety $K \backslash G$ is; this is also known as being \emph{absolutely spherical}. We single out two motivating families of such distributions.
\begin{enumerate}
	\item The notion of matrix coefficients of $\pi$ generalizes to the relative case: given
	\[ \eta \in \EuScript{N}_\pi := \Hom_{G(\R)}(\pi, C^\infty(Z(\R))) \quad \text{(continuous Hom)}, \]
	the space $\eta(V_\pi)$ consists of $\mathcal{Z}(\mathfrak{g})$-finite $C^\infty$-functions on $Z(\R)$. Let $K = G^\theta$ where $\theta$ is a Cartan involution of $G$, so that $K(\R)$ is maximal compact in $G(\R)$. If we consider only $K(\R)$-finite vectors in $\pi$, the generalized matrix coefficients are $\mathfrak{k}$-finite as well. These coefficients are the subject matter of relative harmonic analysis over $\R$; see \cite{KKS17} and the references therein. Unsurprisingly, differential equations with regular singularities entered there.
	
	Note that the space $\EuScript{N}_\pi$ differs from that in \cite[\S 4.1]{Li18} where one considered $C^\infty$ half-densities instead.
	
	\item Let $H_i \subset G$ be spherical subgroups and $\chi_i: H_i(\R) \to \CC^\times$ be smooth characters ($i = 1, 2$). The \emph{relative characters} are certain $\mathcal{Z}(\mathfrak{g})$-finite distributions on $G(\R)$ which are left $(H_1, \chi_1)$-equivariant and right $(H_2, \chi_2)$-equivariant. Specifically, Let $\phi_1$ (resp.\ $\phi_2$) be a continuous $(H_1(\R), \chi_1)$-equivariant (resp.\ $(H_2(\R), \chi_2^{-1})$-equivariant) linear functional of $V_\pi$ (resp.\ $V_{\check{\pi}}$). The corresponding relative character is
	\[ \Theta_{\phi_1, \phi_2}: f \mapsto \lrangle{\phi_1, \pi(f) \phi_2 }, \quad f \in C^\infty_c(G(\R)). \]
	They appear in the local Archimedean components in the spectral side of relative trace formula. Endowing $Z := G$ with the right $G^{\mathrm{op}} \times G$-action $x \xmapsto{(a,b)} axb$ and taking $K := H_1^{\mathrm{op}} \times H_2$, we may regard relative characters as $\mathcal{Z}(\mathfrak{g}^{\mathrm{op}} \times \mathfrak{g})$-finite $(K(\R), (\chi_1, \chi_2^{-1}))$-equivariant distributions on $Z(\R)$, thereby fitting into the previous framework.
\end{enumerate}

The modern theory of algebraic differential systems is phrased in terms of $\mathscr{D}$-modules. Any distribution on $Z(\R)$ generates a $\mathscr{D}_Z$-module, and taking complexification yields a $\mathscr{D}_{Z_{\CC}}$-module. The \emph{regular holonomic} $\mathscr{D}$-modules generalize the systems with regular singularities, and they are related to perverse sheaves on complex analytic manifolds via the Riemann--Hilbert correspondence. For example, Harish-Chandra's differential system for eigendistributions are studied in \cite{HK84} from this perspective. In the recent work \cite{AGM16}, the relative characters are shown to be holonomic. The matrix coefficients in the group case are also related to the wonderful compactifications in \cite{BZG19} using the language of $\mathscr{D}$-modules.

\subsection*{Main results}
Let $Z$ be a spherical homogeneous $G$-variety, $Z(\R) \neq \emptyset$ and $K \subset G$ be a spherical subgroup as alluded to above. We call a $\mathscr{D}_Z$-module regular holonomic if its complexification is. The aim of this article is to show that a large class of $\mathscr{D}_Z$-modules with suitable equivariant or monodromic structures are regular holonomic. This includes the $\mathscr{D}_Z$-modules generated by
\begin{enumerate}
	\item the relative characters $\Theta_{\phi_1, \phi_2}$ (with $Z = G$ and $K = H_1^{\mathrm{op}} \times H_2 \subset G^{\mathrm{op}} \times G$), assuming that the differential of $\chi_i$ factors through the maximal reductive quotient of $\mathfrak{h}_i$, for $i=1,2$;
	\item the $K(\R)$-finite generalized matrix coefficients (with $K = G^\theta$) on $Z(\R)$.
\end{enumerate}
This strengthens the holonomicity of relative characters proven in \cite{AGM16}. Specifically, in Theorem \ref{prop:L-regularity} and Corollary \ref{prop:admissible-regular}, we will prove the regularity for \emph{$K$-admissible} $\mathscr{D}_Z$-modules, as explicated below.

Let $K \subset G$ be a subgroup. We say a character $\chi: \mathfrak{k} \to \CC$ between Lie algebras is \emph{reductive} if $\chi$ factors through the maximal reductive quotient of $\mathfrak{k}$; we say a smooth character $K(\R) \to \CC^\times$ is reductive if its differential is reductive. A $\mathscr{D}_Z$-module $\mathcal{M}$ is said to be $K$-admissible (Definition \ref{def:K-admissible}) if it is generated by a $D_Z$-module $M$, where $D_Z := \Gamma(Z, \mathscr{D}_Z)$, such that
\begin{itemize}
	\item $M$ is finitely generated over $D_Z$,
	\item each element of $M$ is $\mathcal{Z}(\mathfrak{g})$-finite, i.e.\ $M$ is locally $\mathcal{Z}(\mathfrak{g})$-finite,
	\item $\mathcal{M}$ carries a $(K, \chi)$-monodromic structure  (see Definition \ref{def:equivariance}) for some reductive character $\chi: \mathfrak{k} \to \CC$.
\end{itemize}

The definition is global in the sense that it only depends on $M$. The aforementioned monodromic structure can be regarded as a twisted variant of $K$-equivariance; see \cite{BB93, FG10}. If the $(K, \chi)$-monodromic structure is weakened to local $\mathfrak{k}$-finiteness, the resulting notion is called $\mathfrak{k}$-admissibility (Definition \ref{def:k-admissible}).

The notion of $K$-admissibility is inspired by Ginzburg's works \cite{Gin89, Gin93} which consider the case $Z = G/K$ for a symmetric subgroup $K \subset G$. One may imagine that the present work is a direct generalization of \textit{loc.\ cit.} to two spherical subgroups $H, K$ that are not necessarily equal nor symmetric. The regularity is obtained by the same strategy: we pass to a doubled basic affine space of $G$ via the horocycle transform (also known as Harish-Chandra transform), then apply the results à la Beilinson--Bernstein, for which we refer to \cite[\S 2.5]{FG10}. Nonetheless, there are also some differences.

\begin{itemize}
	\item The holonomicity for $\mathfrak{k}$-admissible $\mathscr{D}_Z$-modules is proved in \cite{Gin89} by a parity argument for symmetric subgroups. We prove this in Corollary \ref{prop:holonomic} for all spherical subgroups by applying the same criterion from \textit{loc.\ cit.} twice, with the help of Springer resolutions. This is inspired by \cite{AGM16}.
	\item We do not study the local characterization of admissible modules as done in \cite[Theorem 1.4.2 (ii) $\implies$ (i)]{Gin89}, so the analogues of \cite[\S 3.4]{Gin89} are not needed.
	\item We work with $(K, \chi)$-monodromic $\mathscr{D}_Z$-modules (an extra structure on $\mathscr{D}$-modules), whereas \cite{Gin89} considered locally $\mathfrak{k}$-finite ones (a property of $\mathscr{D}$-modules). The permanence of $(K, \chi)$-monodromicity under various operations is easier to assure.
	\item The reductivity of $\chi$ is necessary in the proof; see Remark \ref{rem:reductive-character} and the discussion below on $\Theta_{\phi_1, \phi_2}$.
\end{itemize}

The result on regularity is directly applicable to relative characters whenever $\chi_i: H_i(\R) \to \CC^\times$ is reductive for $i=1,2$. As for generalized matrix coefficients, we will actually prove that for any Harish-Chandra module $V$, its \emph{localization} on $Z$
\[ \mathrm{Loc}_Z(V) := \mathscr{D}_Z \dotimes{U(\mathfrak{g})} V \]
is generated by the $K$-admissible $D_Z$-module $D_Z \otimes_{U(\mathfrak{g})} V$; here $\chi$ is the trivial character. Taking $V = V_\pi^{K(\R)\text{-fini}}$ for some SAF representation $\pi$, the $K(\R)$-finite generalized matrix coefficients of $\pi$ appear in subquotients of $\mathrm{Loc}_Z(V)$, therefore generate regular holonomic submodules. For further discussions on the localization functor, we refer to \cite{BZG19}.

Note that in the case of relative characters $\Theta_{\phi_1, \phi_2}$, the reductivity assumption on $\chi_1, \chi_2$ excludes the \emph{Whittaker-induced} case, for example when $G$ is quasi-split,  $H_1 = H_2 = U$ is maximal unipotent and $\chi_1 = \chi_2^{-1}$ is a non-degenerate character of $U(\R)$; we refer to \cite{Gin18} for a description of the resulting Whittaker category of $D_G$-modules in terms of nil-Hecke algebras.

Without the reductivity of $\chi_1, \chi_2$, regularity fails according to the final paragraph of Example \ref{eg:relative-characters}, and we can only conclude from $\mathfrak{h}_1^{\mathrm{op}} \times \mathfrak{h}_2$-admissibility that $\Theta_{\phi_1, \phi_2}$ generates a holonomic $D_G$-module, which is already proven in \cite{AGM16}.

\subsection*{Applications}
We give only the simplest consequences of regularity to illustrate its usage. For results which can be deduced by holonomicity alone, we refer to \cite{AGM16}.

Functions, distributions or hyperfunctions (in Sato's sense) on $Z(\R)$ generating a regular holonomic $\mathscr{D}_Z$-module have a quite rigid structure; we refer to \cite[III.1]{Del70} \cite[IX]{Ph11} for further discussions. Let us begin with the simplest properties.
\begin{enumerate}
	\item Suppose that a hyperfunction $u$ generates a regular holonomic $\mathscr{D}_Z$-module, for example when $D_Z \cdot u$ is a subquotient of a $K$-admissible module. First, by holonomicity, there exists a Zariski-open dense $U \subset Z$ on which the $\mathscr{D}_Z$-module is an integrable connection. Then $u|_{U(\R)}$ is analytic. In some cases it is easy to write $U$ down. This is indeed the case for Labesse's twisted space (Example \ref{eg:group-case-U}), which is the main subject of twisted harmonic analysis.

	\item Secondly, in this case it is well-known that $u$ is automatically a distribution; on the other hand, if $u$ is $C^\infty$ then it is automatically analytic (Theorem \ref{prop:hyperfcn-dist}).

	\item Variant: Suppose that $K = G^\theta$ and the hyperfunction $u$ generates a subquotient of a $\mathfrak{k}$-admissible $D_Z$-module. Elliptic regularity theorem implies that $u$ is always analytic, even when $Z$ is non-spherical (Proposition \ref{prop:elliptic-regularity}).
\end{enumerate}

The remaining applications concern growth estimates. It is well-known that $u$ is of at most polynomial growth on the smooth locus $U$, but we have to recast this into a convenient form. Definition--Proposition \ref{def:moderate} provides a notion of moderate growth of $u$ at infinity. Loosely speaking, this means that $p^a u|_{U(\R)} = O(1)$ for any reasonable function $p: U(\R) \to \R_{> 0}$ that decays to zero at infinity, where $a > 0$ depends on $p$ and $u$; the ``infinity'' here is defined using any smooth compactification $U \hookrightarrow X$. After reducing to the case where $X \smallsetminus U$ has normal crossings, the moderate growth at infinity for $u$ (Theorem \ref{prop:growth-estimate}) follows readily from the standard estimates from Deligne \cite{Del70}. A flexible framework for such arguments is provided by \emph{subanalytic geometry}, in particular by Łojasiewicz's inequality recalled in Theorem \ref{prop:inequality}.

The weakness of these growth estimates is the implicit exponent $a$. Take the character $\Theta_\pi$ of an SAF representation $\pi$ for example. Our general result asserts that $|D^G|^a \Theta_\pi$ is locally bounded, where $D^G$ is the Weyl discriminant on $G$; on the other hand, Harish-Chandra obtained this for $a = \frac{1}{2}$.

When applied to generalized matrix coefficients, our ``soft'' method furnishes an estimate that is akin to \cite[Theorem 7.2]{KKS17}, but without any information on the exponent; see Theorem \ref{prop:coeff-estimate} and Corollary \ref{prop:polar-estimate}. Since those results are also easy consequences of the moderate growth of SAF representations, we omit their proofs. 

%Nonetheless, our ``soft'' method is capable of reproving the estimate in \cite[Theorem 7.2]{KKS17} for generalized matrix coefficients (Theorem \ref{prop:coeff-estimate}, Corollary \ref{prop:polar-estimate}); the short argument is based on Łojasiewicz's inequality together with the theory of toroidal embeddings \cite{KK16} over $\R$.

Incidentally, we also prove in Proposition \ref{prop:Hom-vs-Loc} that $\Hom_{\CC}(V/\mathfrak{h}V, \CC)$ is finite-dimensional whenever $H \subset G$ is a spherical subgroup and $V$ is a Harish-Chandra module. It implies that $\EuScript{N}_\pi$ is finite-dimensional. Although these results have been proven in stronger forms, see \cite{AGKL16}, our regularity-based proof can express $\Hom_{\CC}(V/\mathfrak{h}V, \CC)$ in terms of the stalks of the solution complex of $\mathrm{Loc}_Z(V)$.

\subsection*{Structure of this article}
The first part of this article aims at regularity. In \S\ref{sec:monodromic} we collect and review the required notions of monodromic $\mathscr{D}$-modules from \cite{BB93, FG10}, together with several instances for later use. In \S\ref{sec:holonomic}, the notion of $\mathfrak{k}$-admissible modules is defined, and we show their holonomicity when $\mathfrak{k}$ is a spherical subalgebra, by invoking Ginzburg's criterion. The \S\ref{sec:horocycle} is a review of the horocycle correspondence, which is used in \S\ref{sec:proof-regularity} to prove the regularity of $K$-admissible modules.

The second part concerns applications. The \S\ref{sec:subanalytic} and \S\ref{sec:growth} introduce some vocabularies from subanalytic geometry, in order state the notion of moderate growth at infinity. This is then related to solutions of regular holonomic systems in \S\ref{sec:Deligne}, following Deligne's work. The \S\ref{sec:app} presents some immediate applications of regularity to harmonic analysis, including the basic examples and an estimate on admissible distributions. Finally, \S\ref{sec:coeff} is devoted to the special case of generalized matrix coefficients of an SAF representation on homogeneous spherical varieties.

\subsection*{Acknowledgements}
The author is grateful to Bernhard Krötz and Gang Liu for helpful discussions on \S\ref{sec:coeff}. Thanks also go to the anonymous referees for their meticulous reading and refreshing comments. This work is supported by NSFC-11922101.

\subsection*{Conventions}

\begin{itemize}
	\item Real manifolds in this article are equi-dimensional, but need not be connected. Unless otherwise specified, $C^\infty$ functions on a real manifold and continuous functions on a topological space are $\CC$-valued.

	The dual of a vector space $V$ is denoted by $V^\vee$. The underlying vector space of a representation $\pi$ is denoted as $V_\pi$.
	
	When it is necessary to distinguish the derived functors from their non-derived versions, or to indicate their cohomologies, we use the prefix $\Lder$ (resp.\ $\Rder$) to denote the left (resp.\ right) derived ones, such as $\iHom$.

	\item Let $A$ be a ring, or more generally a ring object in a topos. We denote by $A\dcate{Mod}$ the category of left $A$-modules. For any $A$-module $M$, write $\Sym(M)$ and $\bigwedge M$ for its symmetric and exterior algebras, respectively. When $A$ is an algebra over a field $\Bbbk$, we say $M$ is locally finite under $A$ if every $m \in M$ is contained in an $A$-submodule which is finite-dimensional over $\Bbbk$.

	\item Let $\Bbbk$ be a field. By a $\Bbbk$-variety we mean an integral, separated scheme of finite type over $\Spec(\Bbbk)$. If $\Bbbk'$ is an extension of $\Bbbk$, we write $Z_{\Bbbk'} := Z \dotimes{\Bbbk} \Bbbk'$ for any $\Bbbk$-variety $Z$. The set of $\Bbbk$-points of $Z$ is denoted by $Z(\Bbbk)$. The sheaf of regular functions is denoted by $\mathscr{O}_Z$, and $\mathscr{O}_{Z,x}$ is the local ring at $x$.
	
	The cotangent bundle of a smooth $\Bbbk$-variety $Z$ is denoted by $T^* Z$. For a subvariety $W \subset Z$ we denote by $T^*_W Z$ its conormal bundle.
	
	If $Z$ is a $\CC$-variety, $Z^{\ana}$ will denote its analytification. Same convention for $\mathscr{O}_Z$-modules.

	\item Group objects in the category of $\Bbbk$-varieties are called $\Bbbk$-groups. Subgroups of $\Bbbk$-groups are understood as closed $\Bbbk$-subgroups. The opposite group (resp.\ derived subgroup, identity connected component, unipotent radical) of $G$ is denoted by $G^{\mathrm{op}}$ (resp.\ $G_{\mathrm{der}}$, $G^\circ$, $R_u(G)$); the same convention on $G^\circ$ pertains to Lie groups as well.
	
	For any affine $\Bbbk$-group $H$, define the additive groups $X^*(H) := \Hom(H, \Gm)$ and $X_*(H) := \Hom(\Gm, H)$. For $\tau \in \Aut(G)$, denote by $G^\tau$ the fixed locus of $\tau$ in $G$.

	\item Unless otherwise specified, $\Bbbk$-groups act on $\Bbbk$-varieties are on the right, written as $(x, g) \mapsto xg$; accordingly, groups and Lie algebras act on the left of function spaces. The stabilizer of a point $x$ under $G$ is denoted by $\Stab_G(x)$. When an affine $\Bbbk$-group $G$ acts on a normal $\Bbbk$-variety $Z$, we say $Z$ is a $G$-variety; when $G$ acts transitively, $Z$ is said to be a homogeneous $G$-variety. If $H \subset G$ is a subgroup and $X$ is an $H$-variety, we define the quotient
	\[  X \utimes{H} G := X \times G \big/ (xh, h^{-1} g) \sim (x, g), \quad h \in H \]
	which exists as a $G$-variety under mild conditions; see \cite[Theorem 2.2]{Ti11}. Denote the image of $(x, g)$ in $X \utimes{H} G$ as $[x,g]$.

	\item The Lie algebra of a $\Bbbk$-group is denoted as $\mathfrak{g} := \Lie G$, and its dual by $\mathfrak{g}^*$. The center of the universal enveloping algebra $U(\mathfrak{g})$ is denoted by $\mathcal{Z}(\mathfrak{g})$. By a character of $\mathfrak{g}$, we mean a homomorphism of Lie algebras $\mathfrak{g} \to \Bbbk$ (i.e.\ homomorphisms of $\Bbbk$-algebras $U(\mathfrak{g}) \to \Bbbk$), which is automatically zero on $[\mathfrak{g}, \mathfrak{g}]$. A character of $\mathfrak{g}$ is called \emph{reductive} if it factors through the maximal reductive quotient. The adjoint action of $G$ on $\mathfrak{g}$, $\mathfrak{g}^*$ or $G$ itself is denoted as $\Ad$.

	\item For a field $\Bbbk$ of characteristic zero and a smooth $\Bbbk$-variety, $\mathscr{D}_Z$ denotes the sheaf (actually étale-local) of algebraic differential operators on $Z$, and $D_Z := \Gamma(Z, \mathscr{D}_Z)$; the stalk at $x$ is denoted by $\mathscr{D}_{Z, x}$. The same rule applies to modules: $\mathscr{D}_Z$-modules will be denoted by symbols like $\mathcal{M}$, and $D_Z$-modules by $M$, and so forth. We will only consider left $\mathscr{D}_Z$-modules.
	
	Integrable connections will be understood in the algebraic sense. The analytification of a $\mathscr{D}_Z$-module $\mathcal{M}$ is written as $\mathcal{M}^{\ana}$. For affine $Z$, we will switch freely between $\mathscr{D}_Z$-modules and $D_Z$-modules by using the functor $\Gamma(Z, \cdot)$.
\end{itemize}

\section{Equivariant and monodromic \texorpdfstring{$\mathscr{D}$}{D}-modules}\label{sec:monodromic}
Let $\Bbbk$ be a field of characteristic zero with algebraic closure $\overline{\Bbbk}$. For a smooth $\Bbbk$-variety $Z$, the formation of $\mathscr{D}_Z$ and $D_Z := \Gamma(Z, \mathscr{D}_Z)$ is compatible with change of base field $\Bbbk$, and we will mostly be concerned with the case when $\Bbbk = \overline{\Bbbk}$.

\begin{example}
	Let $\Bbbk = \R$ and let $Z$ be a smooth $\R$-variety. In this case $Z(\Bbbk)$ is Zariski-dense in $Z$ if it is nonempty; see \cite[1.A]{Po14}. Therefore $Z$ is $\R$-dense in the sense of \cite{KK16}. Any $C^\infty$-function $u: Z(\R) \to \CC$ generates a $\mathscr{D}_Z$-module $\mathscr{D}_Z \cdot u$, which in turn gives a $\mathscr{D}_{Z_{\CC}}$-module by base change. The same holds for distributions, or more generally for hyperfunctions on $Z(\R)$.
\end{example}

Let $G$ be an affine $\Bbbk$-group and $Z$ be a $G$-variety. Consider the action morphism $a: Z \times G \to Z$, the projection $\mathrm{pr}_1: Z \times G \to Z$, the morphisms
\begin{gather*}
	p_0, p_1, p_2: Z \times G \times G \to Z \times G, \\
	p_0(x, g, h) = (xg, h), \quad p_1(x, g, h) = (x, gh), \quad p_2(x, g, h) = (x, g),
\end{gather*}
and $i: Z \to Z \times G$ given by $i(x) = (x, 1)$. We have $a p_1 = a p_0$, $\mathrm{pr}_1 p_1 = \mathrm{pr}_1 p_2$, $\mathrm{pr}_1 p_0 = a p_2$ and $a i = \mathrm{pr}_1 i = \identity_Z$. The following notions are well-known, see eg.\ \cite[1.8.5]{BB93} or \cite[2.5]{FG10}.

\begin{definition}\label{def:equivariance}
	Let $Z$ be a smooth $G$-variety. The $G$-action on $Z$ induces a homomorphism $U(\mathfrak{g}) \to D_Z$ of $\Bbbk$-algebras. Consider a $\mathscr{D}_Z$-module $\mathcal{M}$.
	\begin{enumerate}
		\item We say that $\mathcal{M}$ is $G$-equivariant, if it is endowed with an isomorphism of $\mathscr{D}_{Z \times G}$-modules
		\[ \varphi: a^* \mathcal{M} \rightiso \mathrm{pr}_1^* \mathcal{M} =  \mathcal{M} \boxtimes \mathscr{O}_G \]
		subject to the cocycle condition that
		\[\begin{tikzcd}[column sep=small]
			p_1^* a^* \mathcal{M} \arrow[rr, "{p_1^* \varphi}"] \arrow[d, "\simeq"'] & & p_1^* \mathrm{pr}_1^* \mathcal{M} \arrow[d, "\simeq"] \\
			p_0^* a^* \mathcal{M} \arrow[rd, "{p_0^* \varphi}"'] & & p_2^* \mathrm{pr}_1^* \mathcal{M} \\
			& p_0^* \mathrm{pr}_1^* \mathcal{M} \simeq p_2^* a^* \mathcal{M} \arrow[ru, "{p_2^* \varphi}"'] &
		\end{tikzcd} \qquad \begin{tikzcd}
			i^* a^* \mathcal{M} \arrow[r, "{i^* \varphi}"] \arrow[d, "\simeq"'] & i^* \mathrm{pr}_1^* \mathcal{M} \arrow[d, "\simeq"] \\
			\mathcal{M} \arrow[r, "\identity"'] & \mathcal{M}
		\end{tikzcd}\]
		are commutative diagrams.
		\item Let $\chi: \mathfrak{g} \to \Bbbk$ be a character of Lie algebras; let $\mathscr{O}_{G, \chi}$ be the trivial line bundle $\mathscr{O}_G$ equipped with the integrable connection $\nabla_\theta u = \theta u - \chi(\theta) u$ for all $\theta \in \mathfrak{g}$, viewed as a right invariant vector field. Then $\mathscr{O}_{G, \chi}$ is a $\mathscr{D}_G$-module: $\theta$ maps $f \in \mathscr{O}_{G, \chi}$ to $\theta f$ (the usual derivative in $\mathscr{O}_G$) plus $\chi(\theta) f$. We say that $\mathcal{M}$ is $(G, \chi)$-monodromic if it is endowed with an isomorphism of $\mathscr{D}_{Z \times G}$-modules
		\[ \varphi: a^* \mathcal{M} \rightiso \mathcal{M} \boxtimes \mathscr{O}_{G, \chi} \]
		subject to cocycle condition. For trivial $\chi$ we recover the notion of $G$-equivariance.
		\item We say that $\mathcal{M}$ is weakly $G$-equivariant, if the $\varphi$ above is only an isomorphism of $\mathscr{D}_Z \boxtimes \mathscr{O}_G$-modules.
	\end{enumerate}
	Note that if $\chi$ lifts to a character $\tilde{\chi}: G \to \Gm$, we have $\mathscr{O}_{G, \chi'} \rightiso \mathscr{O}_{G, \chi + \chi'}$ by $f \mapsto \tilde{\chi} f$ for any $\chi'$.

	The $G$-equivariant (resp.\ weakly $G$-equivariant, $(G, \chi)$-monodromic) $\mathscr{D}_Z$-modules form an abelian category for any given $\chi$: the morphisms are required to respect $\varphi$.
\end{definition}

If $Z = \{\mathrm{pt}\}$, a $G$-equivariant (resp.\ weakly $G$-equivariant) $\mathscr{D}_Z$-module is nothing but a locally finite algebraic representation of $\pi_0(G)$ (resp.\ $G$) over $\Bbbk$.

In concrete terms, if $\mathcal{M}$ is viewed just as a quasi-coherent sheaf on $Z$, then $\varphi: a^* \mathcal{M} \rightiso \mathrm{pr}_1^* \mathcal{M}$ with cocycle conditions encodes a $G$-equivariance structure on $\mathcal{M}$; see \cite[38.12]{stacks-project}. As a $\mathscr{D}_Z$-module, weak equivariance means that the $G$-action on $\mathcal{M}$ is compatible with the $G$-action on $\mathscr{D}_Z$, namely the transport of structure $D \xmapsto{g} g^{-1} D g$; here $g \in G$ acts on $\mathscr{O}_Z$ by the regular representation $f(x) \mapsto f(xg)$, so $g^{-1} D g$ makes sense in $\mathscr{D}_Z$.

Equivariance as a $\mathscr{D}_Z$-module means that the $U(\mathfrak{g})$-action on $\mathcal{M}$ given by $U(\mathfrak{g}) \to \mathscr{D}_Z$ coincides with that given by the $G$-action on $\mathcal{M}$, i.e.\ $\varphi$ is also $\mathscr{D}_G$-linear.

\begin{lemma}\label{prop:monodromic-finiteness}
	Suppose $\mathcal{M}$ is $(G, \chi)$-monodromic for some $\chi$. Then for each open subset $U \subset Z$ and each $s \in \Gamma(U, \mathcal{M})$, the $\Bbbk$-vector space $U(\mathfrak{g}) \cdot s$ is finite-dimensional.
\end{lemma}
\begin{proof}
	Using the isomorphism $\varphi$ above, the required $U(\mathfrak{g})$-finiteness property is transferred to the case of local sections of $\mathscr{O}_{G, \chi}$ under $\nabla$, which is evident.
\end{proof}

We present several examples for later use.

\begin{example}[Function spaces]\label{eg:function-space}
	Take $\Bbbk = \R$ and let $Z$ be a smooth $K$-variety where $K$ is an affine $\R$-group. Let $V$ be a $\CC$-vector space of $C^\infty$ functions on $Z(\R)$. Suppose that
	\begin{compactitem}
		\item $V$ is stable under the regular representation $u(x) \xmapsto{k} u(xk)$ of $K(\R)$ on $C^\infty(Z(\R))$;
		\item the $K(\R)$-representation $V$ extends to a locally finite, algebraic representation of $K(\CC)$ on $V$.
	\end{compactitem}
	Then the $\mathscr{D}_{Z_{\CC}}$-module generated by $V$ is $K$-equivariant. To see this, we use the ``concrete'' interpretation of equivariance. First, the $K(\R)$-action on $V$ is clearly compatible with its action on $\mathscr{D}_Z$ by transport of structure. The $\mathfrak{k}$-actions from $U(\mathfrak{k}) \to \mathscr{D}_Z$ and that from the action on $V$ also coincide, for similar reason. These compatibilities extend algebraically to $K(\CC)$ since the $K(\R)$-representation on $V$ extends. The formula for $\varphi$ reads:
	\begin{equation}\label{eqn:equivariance-fcn}
		u(xk) = \sum_{i=1}^m u_i(x) f_i(k) \implies P \cdot a^*(u) \xmapsto{\varphi} P \cdot \sum_{i=1}^m u_i(x) \otimes f_i(k)
	\end{equation}
	where $P \in \mathscr{D}_{Z \times K}$, $u, u_i \in V$ and $f_i \in \mathscr{O}_K$.
	
	The same holds for distributions and hyperfunctions on $Z(\R)$ as well.
\end{example}

\begin{remark}\label{rem:unitarian}
	Here is a typical application of Example \ref{eg:function-space}: $G$ is a connected reductive $\R$-group, $Z$ is a smooth $G$-variety, $K = G^\theta$ for some Cartan involution $\theta$ of $G$, and $V \subset C^\infty(Z(\R))$ is an admissible $(\mathfrak{g}, K(\R))$-module with respect to the regular representation $f(x) \xmapsto{g} f(xg)$ of $G(\R)$ on $C^\infty(Z(\R))$. In this case, extendibility of $V$ to a locally finite algebraic $K(\CC)$-representation follows from Weyl's unitarian trick, as locally finite representations of $K(\R)$ are algebraic. Again, the same holds for distributions and hyperfunctions.
\end{remark}

\begin{example}[Relative invariants]\label{eg:relative-invariant}
	Let $Z$ be a smooth $K$-variety as in Example \ref{eg:function-space}. Let $u$ be a $C^\infty$-function (or distribution, hyperfunction) on $Z(\R)$ and let $\chi: \mathfrak{k} \to \CC$ be a character, satisfying
	\[ \forall \theta \in \mathfrak{k}, \quad \theta \cdot u = \chi(\theta) u. \]
	Then $u$ generates a $(K, \chi)$-monodromic $\mathscr{D}_{Z_{\CC}}$-module. Specifically, one takes the isomorphism $\varphi$ to be
	\[ P \cdot a^*(u) \mapsto P \cdot (u \otimes 1), \quad P \in \mathscr{D}_{Z \times G}. \]

	Example \ref{eg:function-space} is not applicable to this scenario when $\chi$ does not come from a character $K \to \Gm$, for example when $K$ is unipotent and $\chi$ is nontrivial. Even when $\chi$ lifts, the formula above differs from \eqref{eqn:equivariance-fcn} by the character of $K$.
\end{example}

\begin{example}[Localizations]\label{eg:localization}
	Let $Z$ be a $G$-variety over a field $\Bbbk$ of characteristic zero, and $K$ be a subgroup of $G$, so that the notion of $(\mathfrak{g}, K)$-module is defined. The localization functor (non-derived) is
	\[ \Loc_Z: U(\mathfrak{g}) \dcate{Mod} \to \mathscr{D}_Z \dcate{Mod}, \quad W \mapsto \mathscr{D}_Z \dotimes{U(\mathfrak{g})} W. \]
	When $V$ is a $(\mathfrak{g}, K)$-module, $\Loc_Z(V)$ acquires a weakly $K$-equivariant structure by letting $k \in K$ acting via
	\[ k \cdot (P \otimes v) = k P k^{-1} \otimes kv, \quad P \in \mathscr{D}_Z, \; v \in V. \]
	This is readily seen to be well-defined. It is actually equivariant: the $K$-action induces an $\mathfrak{k}$-action on $\Loc_Z(V)$, which is
	\[ P \otimes v \mapsto (\theta P - P\theta) \otimes v + P \otimes (\theta v) = (\theta P) \otimes v \]
	for all $\theta \in \mathfrak{k}$ and $P \otimes v \in \Loc_Z(V)$.
\end{example}

Another perspective on monodromic modules from \cite[2.5]{BB93} will be needed. Assume henceforth $\Bbbk = \overline{\Bbbk}$. Let $T$ be a $\Bbbk$-torus and $\pi: \tilde{X} \to X$ be a $T$-torsor; $X$ is smooth. Put $\widetilde{\mathscr{D}} := (\pi_* \mathscr{D}_{\tilde{X}})^T$.
\begin{enumerate}
	\item For any ideal $\mathfrak{a}$ of $\Sym(\mathfrak{t})$ and any $\widetilde{\mathscr{D}}$-module $\mathcal{M}$, write $\mathcal{M}[\mathfrak{a}] \subset \mathcal{M}$ for the subsheaf annihilated by $\mathfrak{a}$, which is seen to be a $\tilde{\mathscr{D}}$-submodule. Every $\xi \in \mathfrak{t}^*$ corresponds to a maximal ideal $\mathfrak{m}_\xi \subset \Sym(\mathfrak{t})$, and we write
	\[ \mathcal{M}_\xi := \mathcal{M}[\mathfrak{m}_\xi], \quad \mathcal{M}_{\tilde{\xi}} := \bigcup_{n \geq 1} \mathcal{M}[\mathfrak{m}_\xi^n]. \]
	Define $\mathcal{M}_{\text{fin}} := \bigcup_{\mathfrak{a}} \mathcal{M}[\mathfrak{a}]$ where $\mathfrak{a}$ ranges over the ideals of finite codimension. Then $\mathcal{M}_{\mathrm{fin}} = \bigoplus_\xi \mathcal{M}_{\tilde{\xi}}$.

	\item Since $\pi$ is affine, the study of $\mathscr{D}_{\tilde{X}}$-modules is the same as that of $\pi_* \mathscr{D}_{\tilde{X}}$-modules. Let $\mathfrak{t}^*_{\Z} \subset \mathfrak{t}$ be the lattice of characters from $X^*(T)$. For any ideal $\mathfrak{a} \subset \Sym(\mathfrak{t})$ and a $\pi_* \mathscr{D}_{\tilde{X}}$-module $\mathcal{N}$, we define the submodule
	\begin{equation*}
		\mathcal{N}[\overline{\mathfrak{a}}] := \sum_{\xi \in \mathfrak{t}^*_{\Z}} \mathcal{N}[\xi^* \mathfrak{a}],
	\end{equation*}
	where
	\begin{equation*}
		\xi^* \in \Aut_{\Bbbk}(\Sym(\mathfrak{h})): \mathfrak{h} \ni \chi \mapsto \xi + \xi(\chi).
	\end{equation*}
	The same recipe above yields, for each $\overline{\xi} \in \mathfrak{t}^*/\mathfrak{t}^*_{\Z}$ one defines
	\[ \mathcal{N}_{\overline{\xi}} \subset \mathcal{N}_{\widetilde{\overline{\xi}}} \subset \mathcal{N}_{\mathrm{fin}} := \bigcup_{\substack{\mathfrak{a}: \text{ideal} \\ \mathrm{codim} < \infty}} \mathcal{N}[\overline{\mathfrak{a}}]. \]
\end{enumerate}

Fix $\xi \in \mathfrak{t}^*$ and let $\overline{\xi}$ be its class modulo $\mathfrak{t}^*_{\Z}$. We are interested in the modules $\mathcal{M}$ (resp.\ $\mathcal{N}$) satisfying
\[ \mathcal{M} = \mathcal{M}_{\mathrm{fin}}, \quad \mathcal{M} = \mathcal{M}_{\tilde{\xi}}, \quad \text{or } \mathcal{M} = \mathcal{M}_{\xi}, \quad \text{(resp.\ $\mathcal{N} = \mathcal{N}_{\mathrm{fin}}$, etc.)} \]
By \cite[2.5.3 and 2.5.4]{BB93}, this gives rise to a diagram of abelian (sub)categories
\begin{equation}\label{eqn:categories} \begin{tikzcd}[column sep=small]
	\mathscr{D}_{\tilde{X}} \dcate{Mod}_{\overline{\xi}} \arrow[phantom, r, sloped, "\subset" description] \arrow[d, xshift=0.3em, "{\pi_*}"] & \mathscr{D}_{\tilde{X}} \dcate{Mod}_{\widetilde{\overline{\xi}}} \arrow[phantom, r, sloped, "\subset" description] \arrow[d, xshift=0.3em] & \mathscr{D}_{\tilde{X}} \dcate{Mod}_{\mathrm{fin}} = \displaystyle\prod_{\overline{\eta}} \mathscr{D}_{\widetilde{X}} \dcate{Mod}_{\overline{\eta}} \arrow[d, xshift=0.3em] \\
	(\pi_* \mathscr{D}_{\widetilde{X}}) \dcate{Mod}_{\overline{\xi}} \arrow[phantom, r, sloped, "\subset" description] \arrow[d, xshift=0.3em, "\rho_{\xi}"] \arrow[u, xshift=-0.3em, "{\pi^{-1}}"] & (\pi_* \mathscr{D}_{\widetilde{X}}) \dcate{Mod}_{\widetilde{\overline{\xi}}} \arrow[phantom, r, sloped, "\subset" description] \arrow[d, xshift=0.3em, "\rho_{\widetilde{\xi}}"] \arrow[u, xshift=-0.3em] & (\pi_* \mathscr{D}_{\widetilde{X}}) \dcate{Mod}_{\mathrm{fin}} = \displaystyle\prod_{\overline{\eta}} (\pi_* \mathscr{D}_{\widetilde{X}}) \dcate{Mod}_{\overline{\eta}} \arrow[u, xshift=-0.3em] \\
	\tilde{\mathscr{D}} \dcate{Mod}_{\xi} \arrow[phantom, r, sloped, "\subset" description] \arrow[u, xshift=-0.3em, "{ (\pi_* \mathscr{D}_{\tilde{X}}) \dotimes{\tilde{\mathscr{D}}} -}"] & \tilde{\mathscr{D}} \dcate{Mod}_{\widetilde{\xi}} \arrow[phantom, r, sloped, "\subset" description] \arrow[u, xshift=-0.3em] & \tilde{\mathscr{D}} \dcate{Mod}_{\mathrm{fin}} = \displaystyle\prod_{\eta} \tilde{\mathscr{D}} \dcate{Mod}_\eta
\end{tikzcd}\end{equation}
in which:
\begin{compactitem}
	\item the categories in the last two rows have just been defined;
	\item the pair $(\pi^{-1}, \pi_*)$ realizes an equivalence between $\mathscr{D}_{\tilde{X}} \dcate{Mod}$ and $(\pi_*\mathscr{D}_{\tilde{X}}) \dcate{Mod}$, and this defines the categories in the first row;
	\item the ``induction'' functor $(\pi_* \mathscr{D}_{\tilde{X}}) \dotimes{\tilde{\mathscr{D}}} -$ also turns out to give equivalences $(\pi_* \mathscr{D}_{\tilde{X}}) \dcate{Mod}_{\widetilde{\overline{\xi}}} \to \tilde{\mathscr{D}} \dcate{Mod}_{\widetilde{\xi}}$ and $(\pi_* \mathscr{D}_{\tilde{X}}) \dcate{Mod}_{\overline{\xi}} \to \tilde{\mathscr{D}} \dcate{Mod}_{\xi}$, with quasi-inverses
	\[ \rho_{\widetilde{\xi}}: \mathcal{N} \mapsto \bigcup_{n \geq 1} \mathcal{N}[\mathfrak{m}_\xi^n], \quad \rho_{\xi}: \mathcal{N} \mapsto \mathcal{N}[\mathfrak{m}_\xi] \]
	respectively.
\end{compactitem}

Let us link the first and the third rows in \eqref{eqn:categories}. According to \cite[1.8.9]{BB93}, the inclusions $\mathscr{O}_{\tilde{X}} \hookrightarrow \mathscr{D}_{\tilde{X}}$ and $\mathscr{O}_X \hookrightarrow \tilde{\mathscr{D}}$ induce
\begin{equation}\label{eqn:N-M}
	\mathscr{O}_{\tilde{X}} \dotimes{\pi^{-1} \mathscr{O}_X} \pi^{-1} \mathcal{M} \rightiso \mathscr{D}_{\tilde{X}} \dotimes{\pi^{-1} \tilde{\mathscr{D}}} \pi^{-1} \mathcal{M} \simeq \pi^{-1}\left( \pi_* \mathscr{D}_{\tilde{X}} \dotimes{\tilde{\mathscr{D}}} \mathcal{M}  \right), \quad \mathcal{M} \in \tilde{\mathscr{D}}\dcate{Mod}.
\end{equation}

Furthermore, $\tilde{\mathscr{D}} \dcate{Mod}_\xi$ is equivalent to $\mathscr{D}_\xi \dcate{Mod}$, where $\mathscr{D}_\xi := \tilde{\mathscr{D}} / \mathfrak{m}_\xi \tilde{\mathscr{D}}$ is the sheaf on $X$ of locally trivial \emph{twisted differential operators} (TDO's) associated with $\xi \in \mathfrak{t}^*$; see \cite[2.1]{BB93}. In this article, we prefer to connect $\tilde{\mathscr{D}} \dcate{Mod}_\xi$ to $(T, \xi)$-monodromic $\mathscr{D}_{\tilde{X}}$-modules as follows.

\begin{proposition}\label{prop:finite-to-monodromic}
	Fix $\xi \in \mathfrak{t}^*$. Every object $\mathcal{N}$ of $\mathscr{D}_{\tilde{X}} \dcate{Mod}_{\overline{\xi}}$ carries a canonical $(T, \xi)$-monodromic structure. This realizes an equivalence of abelian categories
	\[ \left\{ (T, \xi)\text{-monodromic modules} \right\} \leftrightharpoons \tilde{\mathscr{D}}\dcate{Mod}_\xi \simeq \mathscr{D}_\xi \dcate{Mod}. \]
\end{proposition}
\begin{proof}
	Our routine arguments are built on the mutually quasi-inverse functors
	\[\begin{tikzcd}[column sep=small]
		\pi^{-1} \left( \pi_* \mathscr{D}_{\tilde{X}} \dotimes{\tilde{\mathscr{D}}} - \right): \tilde{\mathscr{D}}\dcate{Mod} \arrow[r, yshift=0.3em] & \left\{ \text{weakly $T$-equivariant $\mathscr{D}_{\tilde{X}}$-modules} \right\} \arrow[l, yshift=-0.3em] : \pi_*(-)^T .
	\end{tikzcd}\]
	This is the content of \cite[1.8.10]{BB93}, where the weakly $T$-equivariant modules are called weak $(\mathscr{D}_{\tilde{X}}, T)$-modules (see 1.8.5 of \text{loc.\ cit.}).

	Define the action and projection morphisms $a, \mathrm{pr}_1: \tilde{X} \times T \to \tilde{X}$. Let $\mathcal{N}$ be realized as $\mathscr{O}_{\tilde{X}} \dotimes{\pi^{-1} \mathscr{O}_X} \pi^{-1} \mathcal{M}$ via \eqref{eqn:N-M}, where $\mathcal{M}$ is a $\tilde{\mathscr{D}}$-module. As $\pi a = \pi \mathrm{pr}_1$, the isomorphism $\varphi$ in Definition \ref{def:equivariance} can be explicitly given using
	\begin{align*}
		a^* \mathcal{N} & \simeq \mathscr{O}_{\tilde{X} \times T} \dotimes{a^{-1} \pi^{-1} \mathscr{O}_X} a^{-1} \pi^{-1} \mathcal{M}, \\
		\mathrm{pr}_1^* \mathcal{N} & \simeq \mathscr{O}_{\tilde{X} \times T} \dotimes{\mathrm{pr}_1^{-1} \pi^{-1} \mathscr{O}_X} \mathrm{pr}_1^{-1} \pi^{-1} \mathcal{M} \\
		& = \pi^{-1} \mathcal{M} \boxtimes \mathscr{O}_T.
	\end{align*}

	We have to show that weak equivariance structure is $(T, \xi)$-monodromic when $\mathcal{M} \in \tilde{\mathscr{D}} \dcate{Mod}_{\xi}$ through the isomorphisms above. We have $\mathscr{D}_T = \mathscr{O}_T \cdot \Sym(\mathfrak{t})$. Given $\theta \in \mathfrak{t}$, it acts on $a^* \mathcal{N}$ by Leibniz rule (see \cite[\S 1.3]{HTT08}); the result is the sum of
	\begin{compactenum}
		\item the effect of $\theta$ on $\mathscr{O}_{\tilde{X} \times T} = \mathscr{O}_{\tilde{X}} \boxtimes \mathscr{O}_T$ through the second slot, and
		\item the effect on $a^{-1} \pi^{-1} \mathcal{M}$: note that $\theta$ induce an operator in $\mathscr{D}_{\tilde{X}}$ which actually comes from $\pi^{-1} \tilde{\mathscr{D}}$, so the $\theta$-action on $a^{-1} \pi^{-1} \mathcal{M}$ equals the scalar $\xi(\theta)$.
	\end{compactenum}
	The same applies to the $\theta$-action on $\mathrm{pr}_1^*$, except that the effect on $\mathrm{pr}_1^{-1} \pi^{-1} \mathcal{M}$ is trivial. To make $\varphi: a^* \mathcal{N} \to \mathcal{N} \boxtimes \mathscr{O}_{T, \xi}$ commute with $\mathscr{D}_T$-action, one replaces the $\mathscr{O}_T$ in $\mathrm{pr}_1^* \mathcal{N}$ by $\mathscr{O}_{T, \xi}$.
	
	Conversely, consider a $(T, \xi)$-monodromic $\mathcal{N}$. Being weakly $T$-equivariant, it is canonically isomorphic to $\pi^{-1} \left( \pi_* \mathscr{D}_{\tilde{X}} \dotimes{\tilde{\mathscr{D}}} \mathcal{M} \right)$ where $\mathcal{M} := \pi_*(\mathcal{N})^T$ is a $\tilde{\mathscr{D}}$-module. Let $\theta := \gamma'(1) \in \mathfrak{t}$ where $\gamma \in X_*(T)$. The $\theta$-action on $\mathcal{M} \subset \pi_* \mathcal{N}$ is determined from $\mathcal{N}$: it is the sum of
	\begin{compactenum}
		\item the derivative at $t=1$ of the $\gamma(t)$-action, which is $0$ since $\mathcal{M} = \pi_*(\mathcal{N})^T$, and
		\item the scalar multiplication by $\xi(\theta)$, since $\mathcal{N}$ is monodromic. 
	\end{compactenum}

	By varying $\theta$ (or $\gamma$), we see $\mathcal{M} = \mathcal{M}_\xi$, hence $\mathcal{M} \in \tilde{\mathscr{D}}\dcate{Mod}_\xi$ as required. Finally, the equivalence with $\mathscr{D}_\xi \dcate{Mod}$ has already been remarked.
\end{proof}

Now consider a general field $\Bbbk$ of characteristic zero and an affine $\Bbbk$-group $G$.

\begin{definition}\label{def:equivariant-derived}
	For a smooth $G$-scheme $Z$ and a character $\chi: \mathfrak{g} \to \Bbbk$, we denote the bounded equivariant derived category of $(G, \chi)$-monodromic $\mathscr{D}_Z$-modules as $\cate{D}^{\mathrm{b}}_{G, \chi}(Z)$: this is a triangulated category with a $t$-structure satisfying the following properties.
	\begin{itemize}
		\item The heart of $\cate{D}^{\mathrm{b}}_{G, \chi}(Z)$ is equivalent to the abelian category of $(G, \chi)$-monodromic $\mathscr{D}_Z$-modules.
		\item For any subgroup $L \subset G$ with $\eta := \chi|_{\mathfrak{l}}$, we have the forgetful functor $\mathbf{oblv}: \cate{D}^{\mathrm{b}}_{G, \chi}(Z) \to \cate{D}^{\mathrm{b}}_{L, \eta}(Z)$.
		\item The functors $\mathbf{oblv}$ are $t$-exact, and induce the usual forgetful functors on cohomologies (i.e.\ forgetting the monodromic structure).
		\item The usual operations on complexes of $\mathscr{D}$-modules (such as $f^!$, $f_*$, etc.) relative to $G$-equivariant morphisms lift to the monodromic setting, with the caveat that $(G, \chi)$ and $(G, -\chi)$ are exchanged under duality (cf.\ Definition \ref{def:equivariance}); we will not make direct use of the duality functor. All these operations commute with forgetful functors. The usual adjunction relations also hold in this generality.
	\end{itemize}
	When $\chi$ is trivial, we obtain the $G$-equivariant derived category $\cate{D}^{\mathrm{b}}_G(Z)$ and the forgetful functor to $\cate{D}^{\mathrm{b}}(Z)$. When $G = \{1\}$, we recover $\cate{D}^{\mathrm{b}}(Z)$.
\end{definition}

A few remarks are in order. The classical accounts on $\mathscr{D}$-modules often impose quasi-projectivity on the varieties. This constraint can be safely removed in view of recent theories, such as that of crystals \cite{GR14}; see also \cite[Chapter 4]{GR17b}. When $\chi$ is trivial, $\cate{D}^{\mathrm{b}}_G(Z)$ is originally defined in \cite[\S 4]{BL94}, and the six operations in this framework are in \cite[Theorem 3.4.1]{BL94}. This theory can also be understood in terms of $\mathscr{D}$-modules (more accurately: crystals) on quotient stacks $[Z/G]$, within the formalism of stable $\infty$-categories. Passing to the homotopy category yields the required equivariant derived categories.

For example, $\mathbf{oblv}: \cate{D}^{\mathrm{b}}_G(Z) \to \cate{D}^{\mathrm{b}}(Z)$ is ``locally the same'' as pull-back of $\mathscr{D}$-modules along various $G$-torsors $P \to S$ such that $P$ maps equivariantly to $Z$; this operation is clearly $t$-exact and induces the usual pull-back on cohomologies since $G$ is smooth. The case with nontrivial $G$-monodromy $\chi$ is explained in \cite[\S 6.5]{GR14} by employing the formalism of TDO's, and this includes our setting of $(G, \chi)$-monodromic modules by Proposition \ref{prop:finite-to-monodromic}, by considering the $G/G_{\mathrm{der}}$-torsor $[Z/G^{\mathrm{der}}] \to [Z/G]$.

We will only make mild use of equivariant derived categories as a blackbox in \S\ref{sec:proof-regularity}, and leave these issues aside.

\section{\texorpdfstring{$\mathfrak{k}$}{k}-admissible \texorpdfstring{$D$}{D}-modules: holonomicity}\label{sec:holonomic}
Fix an algebraically closed field $\Bbbk$ of characteristic zero and a connected reductive $\Bbbk$-group $G$. In what follows, $Z$ will be a homogeneous $G$-variety over $\Bbbk$.

\begin{definition}[V.\ Ginzburg {\cite[Definition 1.2]{Gin89}}]\label{def:k-admissible}
	Let $K$ be a subgroup of $G$. A $D_Z$-module $M$ is called $\mathfrak{k}$-\emph{admissible} if
	\begin{compactitem}
		\item $M$ is finitely generated over $D_Z$;
		\item for every $m \in M$, the dimensions of $U(\mathfrak{k}) \cdot m$ and $\mathcal{Z}(\mathfrak{g}) \cdot m$ are both finite --- in other words, $M$ is locally finite under $U(\mathfrak{k})$ and $\mathcal{Z}(\mathfrak{g})$.
	\end{compactitem}
\end{definition}

Denote the $\mathscr{D}_Z$-module generated by $M$ as $\mathcal{M}$. Quotients and finitely generated submodules of a $\mathfrak{k}$-admissible $D_Z$-module are still admissible.

\begin{remark}
	The definition of $\mathfrak{k}$-admissibility is of a global nature. In \textit{loc.\ cit.}, $Z$ is assumed to be affine so that the global sections functor $\Gamma: \mathscr{D}_Z\dcate{Mod} \to D_Z\dcate{Mod}$ is an equivalence. Our $\mathscr{D}_Z$-modules $\mathcal{M}$ are globally generated by construction, and the properties of $\mathcal{M}$ such as holonomicity, etc.\ will be accessed through $M$.
\end{remark}

\begin{remark}\label{rem:M0}
	A $D_Z$-module $M$ is $\mathfrak{k}$-admissible if and only if $M$ is generated by a $\Bbbk$-subspace $M_0$ such that $M_0$ is finite-dimensional and closed under the actions of both $U(\mathfrak{k})$ and $\mathcal{Z}(\mathfrak{g})$.
\end{remark}

Consider the cotangent bundle $T^* Z \xrightarrow{\pi} Z$. Every $v \in \mathfrak{g}$ induces a vector field $\xi_v$ on $Z$. One can evaluate $\xi_v$ at any point of $T^* Z$, giving rise to the \emph{moment map}
\[ \bm{\mu} = \bm{\mu}_{\mathfrak{g}}: T^* Z \to \mathfrak{g}^*. \]
Recall that $T^* Z$ carries a natural symplectic structure and a $G$-action. For a smooth $\Bbbk$-variety $X$ carrying a symplectic structure and a closed subvariety $Y \subset X$, we say $Y$ is co-isotropic (resp.\ isotropic, Lagrangian) if $T_y Y$ is a co-isotropic (resp.\ isotropic, Lagrangian) subspace of $T_y X$, at every smooth point $y$ of $Y$. For example, the conormal bundle $T^*_W Z$ is Lagrangian in $T^* Z$ for any locally closed $W \subset Z$. The characteristic variety $\mathrm{Ch}(\mathcal{M})$ of any coherent $\mathscr{D}_Z$-module $\mathcal{M}$ is a conic, co-isotropic closed subvariety of $T^* Z$ (see \cite[Theorem 2.3.1]{HTT08}.) When $\dim \mathrm{Ch}(\mathcal{M}) = \dim Z$, we say $\mathcal{M}$ is \emph{holonomic}. By abuse of notation, $M$ is also said to be holonomic.

Let $\mathcal{N} \subset \mathfrak{g}^*$ be the nilpotent cone. It is the zero locus of all $f \in \Bbbk[\mathfrak{g}^*]^G$ without constant terms.

\begin{proposition}[V.\ Ginzburg {\cite[Lemma 2.1.2]{Gin89}}]\label{prop:Ginzburg-bound}
	Let $K \subset G$ be a subgroup and let $M$ be a $\mathfrak{k}$-admissible $D_Z$-module, then
	\[ \mathrm{Ch}(\mathcal{M}) \subset \bm{\mu}^{-1} \left( \mathcal{N} \cap \mathfrak{k}^\perp \right). \]
\end{proposition}
\begin{proof}
	Let us sketch the arguments in \textit{loc.\ cit.} briefly. Take $M_0$ as in Remark \ref{rem:M0} so that $\mathcal{M} = \mathscr{D}_Z \cdot M_0$. Define the good filtration $F^i \mathcal{M} := \mathscr{D}_Z^i \cdot M_0$ (see \cite[Definition 2.1.2]{HTT08}) where $\mathscr{D}_Z^\bullet$ is the filtration of $\mathscr{D}_Z$ by degrees. Denote by $\mathcal{Z}^+(\mathfrak{g}) \subset \mathcal{Z}(\mathfrak{g})$ is the augmentation ideal, with the filtration induced from $U(\mathfrak{g})$. One checks that each $F^i \mathcal{M}$ is stable under $U(\mathfrak{k})$ and $\mathcal{Z}^+(\mathfrak{g})$. It follows that $\gr_F \mathcal{M}$ is annihilated by $\gr \mathcal{Z}^+(\mathfrak{g})$ and $\mathfrak{k}$ (more precisely, under their images in $\gr \mathscr{D}_Z \simeq \pi_* \mathscr{O}_{T^* Z}$). By considering their zero loci, one can infer that $\Supp(\gr_F \mathcal{M}) \subset \bm{\mu}^{-1}\left( \mathcal{N} \cap \mathfrak{k}^\perp \right)$.
\end{proof}

By choosing $x_0 \in Z(\Bbbk)$ and putting
\[ H := \Stab_G(x_0), \quad H \backslash G \xrightarrow[Hg \mapsto x_0 g]{\sim} Z, \]
we can identify
\[ T_{x_0}^* Z \simeq \mathfrak{h}^\perp, \quad T^* Z \simeq \mathfrak{h}^\perp \utimes{H} G \quad \text{($G$-equivariant)}. \]

The group $G$ acts on $\mathfrak{g}^*$ through the coadjoint action. Writing elements of $\mathfrak{h}^\perp \utimes{H} G$ as equivalence classes $[\omega, g]$, where $\omega \in \mathfrak{h}^\perp$, $g \in G$ and impose the relation $[\omega, hg] = [h^{-1} \omega h, g]$ for all $h \in H$, the moment map becomes
\begin{align*}
\bm{\mu}: \mathfrak{h}^\perp \utimes{H} G & \longrightarrow \mathfrak{g}^* \\
[\omega, g] & \longmapsto g^{-1} \omega g .
\end{align*}

We have the following criterion due to Ginzburg. First, recall that $\mathcal{N}$ is the union of all nilpotent coadjoint orbits in $\mathfrak{g}^*$, which are finite in number. Each coadjoint orbit $\mathcal{O}$ is endowed with the Kirillov--Kostant--Souriau symplectic structure; cf.\ \cite[Proposition 1.1.5]{CG10}. By stipulation, $\emptyset$ is Lagrangian in any smooth symplectic variety.

\begin{proposition}[V.\ Ginzburg]\label{prop:Ginzburg-criterion} % No assumption of quasi-affineness, etc.
	Let $H, K \subset G$ be subgroups and let $Z := H \backslash G$. Define $\bm{\mu}: T^* Z \to \mathfrak{g}^*$ as before. For every nilpotent coadjoint orbit $\mathcal{O} \subset \mathcal{N}$, the following are equivalent:
	\begin{itemize}
		\item $\bm{\mu}^{-1}(\mathcal{O} \cap \mathfrak{k}^\perp)$ is isotropic (resp.\ co-isotropic, Lagrangian) in $T^* Z$;
		\item $\mathcal{O} \cap \mathfrak{h}^\perp$ and $\mathcal{O} \cap \mathfrak{k}^\perp$ are both isotropic (resp.\ co-isotropic, Lagrangian) in $\mathcal{O}$.
	\end{itemize}
	These intersections are set-theoretic, i.e.\ they are reduced schemes.
\end{proposition}
\begin{proof}
	This is \cite[Proposition 1.5.1]{Gin89}, whose proof is in \S 3.1 of \textit{loc.\ cit.}
\end{proof}

\begin{corollary}\label{prop:isotropic-implies-holonomic}
	Suppose that $\mathcal{O} \cap \mathfrak{h}^\perp$ and $\mathcal{O} \cap \mathfrak{k}^\perp$ are both isotropic for every nilpotent coadjoint orbit $\mathcal{O} \subset \mathcal{N}$. Then
	\begin{enumerate}[(i)]
		\item $\bm{\mu}^{-1}(\mathcal{N} \cap \mathfrak{k}^\perp)$ is Lagrangian in $T^* Z$;
		\item all $\mathfrak{k}$-admissible $D_Z$-modules are holonomic.
	\end{enumerate}
\end{corollary}
\begin{proof}
	By the finiteness of nilpotent coadjoint orbits, $\bm{\mu}^{-1}(\mathcal{N} \cap \mathfrak{k}^\perp) = \bigcup_{\mathcal{O}} \bm{\mu}^{-1}(\mathcal{O} \cap \mathfrak{k}^\perp)$ is isotropic in $T^* Z$; in particular its dimension cannot exceed $\dim Z$. Let $M$ be any $\mathfrak{k}$-admissible $D_Z$-module. From $\mathrm{Ch}(\mathcal{M}) \subset \bm{\mu}^{-1}(\mathcal{N} \cap \mathfrak{k}^\perp)$ we see $\mathrm{Ch}(\mathcal{M})$ is isotropic by \cite[Proposition 1.3.30]{CG10}. Hence $\mathrm{Ch}(\mathcal{M})$ is Lagrangian and $M$ is holonomic. This proves (ii). Now
	\[ \dim Z = \dim T^*_Z Z \leq \dim \bm{\mu}^{-1}(\mathcal{N} \cap \mathfrak{k}^\perp) \leq \dim Z \]
	implies that $\dim \bm{\mu}^{-1} (\mathcal{N} \cap \mathfrak{k}^\perp) = \dim Z$, so $\bm{\mu}^{-1} (\mathcal{N} \cap \mathfrak{k}^\perp)$ is Lagrangian, proving (i).
\end{proof}

% Originally an item (iii): there exists a nilpotent $\mathcal{O}$ such that $\mathcal{O} \cap \mathfrak{h}^\perp$ and $\mathcal{O} \cap \mathfrak{k}^\perp$ are both Lagrangian in $\mathcal{O}$.
%  Original proof:
%	Finally, since $\bm{\mu}^{-1} (\mathcal{N} \cap \mathfrak{k}^\perp) = \bigcup_{\mathcal{O}} \bm{\mu}^{-1}(\mathcal{O} \cap \mathfrak{k}^\perp)$ (finite union) is Lagrangian, we must have $\dim \bm{\mu}^{-1}(\mathcal{O} \cap \mathfrak{k}^\perp) \geq \dim Z$ for some $\mathcal{O}$. Hence $\bm{\mu}^{-1}(\mathcal{O} \cap \mathfrak{k}^\perp)$ is Lagrangian since it is also isotropic. It remains to apply Proposition \ref{prop:Ginzburg-criterion} again to deduce (iii).
% Problem: cannot rule out the trivial case $\mathcal{O} = 0$

We are now ready to prove holonomicity of admissible $D$-modules in the spherical case.
\begin{definition}
	A $G$-variety $Z$ is said to be \emph{spherical} if it has an open $B$-orbit, for some (equivalently, any) Borel subgroup $B$ of $G$. A subgroup $H \subset G$ is said to be spherical if $H \backslash G$ is spherical; this property depends only on $\mathfrak{h}$.
	
	For non-algebraically closed fields $\Bbbk$, we say a $G$-variety $Z$ over $\Bbbk$ is spherical if $Z_{\overline{\Bbbk}}$ is. Such $G$-varieties are often called \emph{absolutely spherical}, for example in \cite{KKS17}.
\end{definition}

Let $\mathcal{B}$ denote the flag variety, i.e.\ the $G$-variety of Borel subgroups of $G$. Note that $H \subset G$ is spherical if and only if $\mathcal{B}$ has only finitely many $H$-orbits.

\begin{theorem}\label{prop:isotropy}
	If $K$ is a spherical subgroup of $G$, then $\mathcal{O} \cap \mathfrak{k}^\perp$ is isotropic in $\mathcal{O}$ for every nilpotent coadjoint orbit $\mathcal{O}$, where $\mathcal{O}$ is endowed with the Kirillov--Kostant--Souriau symplectic structure.
\end{theorem}
\begin{proof}
	We may assume $K$ connected. Consider the moment map for the $G$-variety $\mathcal{B}$, denoted as $r$. By fixing a Borel subgroup $B_0$, we have
	\begin{align*}
		B_0 \backslash G & \rightiso \mathcal{B} \\
		B_0 g & \mapsto g^{-1} B_0 g, \\
		r: T^* \mathcal{B} \simeq \mathfrak{b}_0^\perp \utimes{B_0} G & \twoheadrightarrow \mathcal{N} \subset \mathfrak{g}^* \\
		[x, g] & \mapsto g^{-1} xg.
	\end{align*}
	In other words, $r$ is the Springer resolution for $\mathcal{N}$. Now fix a nilpotent coadjoint orbit $\mathcal{O}$. By applying Proposition \ref{prop:Ginzburg-criterion} to the given subgroup $K$, $H := B_0$ and $Z = \mathcal{B}$, we obtain
	\[ r^{-1}(\mathcal{O} \cap \mathfrak{k}^\perp) \text{ is isotropic in } T^* \mathcal{B} \iff \text{so are } \mathcal{O} \cap \mathfrak{b}_0^\perp, \; \mathcal{O} \cap \mathfrak{k}^\perp \;\text{ in } \mathcal{O}. \]
	Note that $\mathcal{O} \cap \mathfrak{b}_0^\perp$ is known to be Lagrangian \cite[Theorem 3.3.7]{CG10}.
	
	It remains to show $r^{-1}(\mathcal{O} \cap \mathfrak{k}^\perp)$ is isotropic. Set $\mathcal{L} := r^{-1}(\mathfrak{k}^\perp)$: it consists precisely of the cotangent vectors of $\mathcal{B}$ that are orthogonal to the vector fields induced by $\mathfrak{k}$. Let $F_1, \ldots, F_k$ be the $K$-orbits in $\mathcal{B}$, so that
	\[ \mathcal{L} = \bigsqcup_{i=1}^k T_{F_i}^* \mathcal{B}. \]
	The $T_{F_i}^* \mathcal{B}$ are defined as in \cite[p.65]{HTT08} and are Lagrangian in $T^* \mathcal{B}$, for $i = 1, \ldots, k$. Hence $\mathcal{L}$ is Lagrangian as well. It follows from \cite[Proposition 1.3.30]{CG10} that $r^{-1}(\mathcal{O} \cap \mathfrak{k}^\perp)$ is isotropic in $T^* \mathcal{B}$, since $r^{-1}(\mathcal{O} \cap \mathfrak{k}^\perp) \subset \mathcal{L}$.
\end{proof}

We remark that the usage of $\mathcal{L}$ in the foregoing arguments is inspired by the proof of \cite[Lemma 2.2]{AGM16}. In the following cases, $\mathcal{O} \cap \mathfrak{h^\perp}$ is even known to be Lagrangian:
\begin{enumerate}
	\item $\mathfrak{h}$ is spherical and solvable, see \cite[Theorem 1.5.7]{CG10};
	\item $\mathfrak{h} = \mathfrak{g}^\theta$ for some involution $\theta$ of $G$, see the proof of \cite[Proposition 3.1.1]{Gin89}.
\end{enumerate}

\begin{corollary}\label{prop:holonomic}
	Let $Z$ be a spherical homogeneous $G$-variety, and let $K$ be a spherical subgroup of $G$. Then every $\mathfrak{k}$-admissible $D_Z$-module $M$ is holonomic. In particular, there is a $K$-invariant Zariski open dense subset $U \subset Z$ on which $\mathcal{M}$ is an integrable connection.
	
	If $G$, $Z$, $K$ are defined over a subfield $\Bbbk_0 \subset \Bbbk$ and $Z(\Bbbk_0) \neq \emptyset$, one can choose $U$ to be defined over $\Bbbk_0$ as well.
\end{corollary}
\begin{proof}
	Choose $x_0 \in Z(\Bbbk)$. Apply Theorem \ref{prop:isotropy} to the spherical subgroups $K$ and $H := \Stab_G(x_0)$, then conclude holonomicity by using Corollary \ref{prop:isotropic-implies-holonomic}. It is well-known that there is an open dense $U \subset Z$ over which $\bm{\mu}^{-1}(\mathcal{N} \cap \mathfrak{k}^\perp)$ (hence $\mathrm{Ch}(\mathcal{M})$) reduces to the zero section, eg.\ \cite[Proposition 3.1.6]{HTT08}. By the equivariance of $\bm{\mu}$, one can replace $U$ by $U \cdot K$ to assume $K$-invariance. The last assertion follows immediately.
\end{proof}

In concrete applications, it is often important to determine the $U$ in Corollary \ref{prop:holonomic}. In the symmetric case $K = G^\theta$ and $Z = K \backslash G$, where $\theta$ is an involution of $G$, we refer to \cite[Proposition 3.5.1]{Gin89} for such a description. Below is another explicit and easier instance.

\begin{example}[Twisted spaces]\label{eg:group-case-U}
	Let us illustrate the description of $U$ by the case of twisted spaces of Labesse; a detailed discussion can be found in \cite[I.3]{HL17}. Take $\Bbbk$ to be a field of characteristic zero, a \emph{twisted space} $\tilde{G}$ under a connected reductive $\Bbbk$-group $G$ is the following data:
	\begin{compactitem}
		\item $\tilde{G}$ is a $G^{\mathrm{op}} \times G$-homogeneous variety, $\tilde{G}(\Bbbk) \neq \emptyset$, with action written as $\gamma \cdot (a, b) = a \gamma b$;
		\item $\tilde{G}$ is simultaneously a $G^{\mathrm{op}}$-torsor and a $G$-torsor, and there exists $\Ad: \tilde{G} \to \Aut(G)$ such that
		\[ \gamma g = \Ad(\gamma)(g) \gamma, \quad \gamma \in \tilde{G}, \; g \in G. \]
	\end{compactitem}
	It follows that $\Ad(a\gamma b) = \Ad(a) \Ad(\gamma) \Ad(b)$. Steinberg's theorem \cite[Théorème I.3.7.1]{HL17} implies that $\Ad(\gamma)$ stabilizes a Borel subgroup over $\overline{\Bbbk}$ for every $\gamma \in \tilde{G}(\overline{\Bbbk})$, therefore $\tilde{G}$ is a spherical $G^{\mathrm{op}} \times G$-variety by Bruhat decomposition. When there exists $\gamma_0$ with $\Ad(\gamma_0) = \identity$, we are reduced to the well-studied ``group case'' $\tilde{G} \simeq G$.

	Let $\ell$ be the absolute rank of $\tilde{G}$ defined in \cite[p.60]{HL17}. For each $\gamma \in \tilde{G}$, define $D^{\tilde{G}}(\gamma)$ to be the coefficient of $X^\ell$ in $\det\left( X - \Ad(\gamma) + 1 \middle| \mathfrak{g} \right)$ where $X$ is an indeterminate. Then $D^{\tilde{G}}$ is a regular function on $\tilde{G}$. Define $\tilde{G}_{\mathrm{reg}} := \{ D^{\tilde{G}} \neq 0 \} \subset \tilde{G}$; the elements thereof are called \emph{regular elements}. This generalizes the notion of regular semisimple elements in $G$. A basic fact is that $\tilde{G}_{\mathrm{reg}}$ is open dense in $\tilde{G}$.
	
	Choose $\gamma_0 \in \tilde{G}(\Bbbk)$ and put $\tau := \Ad(\gamma_0)^{-1} \in \Aut(G)$. Then
	\[ H := \Stab_{G^{\mathrm{op}} \times G}(\gamma_0) = \left\{ (a, b): \gamma_0 \tau(a) b = \gamma_0  \right\} = \left\{ (g, \tau(g)^{-1}) : g \in G \right\}. \]
	Thus $\mathfrak{h} = \Image(\identity, -\tau) \subset \mathfrak{g}^{\mathrm{op}} \times \mathfrak{g}$ and $\mathfrak{h}^\perp = \Image(\identity, \tau) \in \mathfrak{g}^{\mathrm{op}, *} \times \mathfrak{g}^*$, where we write $\tau$ for the induced automorphisms on $\mathfrak{g}$ and $\mathfrak{g}^*$. Summing up:
	\[\begin{tikzcd}[row sep=tiny]
		\mathfrak{g}^* \times G \arrow[r, "\sim"] & \mathfrak{h}^\perp \utimes{H} \left( G^{\mathrm{op}} \times G \right) \arrow[r, "\sim"] & T^* \tilde{G} \arrow[r] & \tilde{G} \\
		(\lambda, g) \arrow[mapsto, r] & {[(\lambda, \tau\lambda), (1, g)]} \arrow[mapsto, rr] & & \gamma_0 g
	\end{tikzcd}\]
	Identify $T^* \tilde{G}$ with $\mathfrak{g}^* \times G$. Then
	\[ \bm{\mu}(\lambda, g) = (\lambda , g^{-1} (\tau\lambda) g) = \left( \lambda, \Ad(g^{-1}) \tau \lambda\right). \]
	
	Next, take the spherical subgroup $K := \left\{(g^{-1}, g): g \in G \right\}$ of $G^{\mathrm{op}} \times G$. We have $\mathfrak{k}^\perp = \{ (\mu, \mu) : \mu \in \mathfrak{g}^* \}$. By the definition of $\tau$,
	\[ \bm{\mu}(\lambda, g) \in \mathfrak{k}^\perp \iff \Ad(g)\lambda = \Ad(\gamma_0)^{-1}(\lambda) \iff \Ad(\gamma_0 g) \lambda = \lambda. \]
	
	Fix an invariant bilinear form $\mathfrak{g} \times \mathfrak{g} \to \Bbbk$ to identify $\mathfrak{g} \simeq \mathfrak{g}^*$. Assume $\gamma := \gamma_0 g$ is regular, then:
	\begin{compactitem}
		\item $\bm{\mu}(\lambda, g) \in \mathfrak{k}^\perp$ is equivalent to $\lambda \in \mathfrak{g}^{\Ad(\gamma)}$, whilst $G^{\Ad(\gamma), \circ}$ is a torus by \cite[Lemme I.3.11.2]{HL17};
		\item $\bm{\mu}(\lambda, g) \in \mathcal{N}$ is equivalent to $\lambda$ being nilpotent.
	\end{compactitem}
	The conjunction of the two properties above is thus $\lambda = 0$. This shows that $\bm{\mu}^{-1}\left( \mathfrak{k}^\perp \cap \mathcal{N} \right)$ reduces to zero section over $\tilde{G}_{\mathrm{reg}}$. Hence we may choose $U := \tilde{G}_{\mathrm{reg}}$.
\end{example}

\section{Review of horocycle correspondence}\label{sec:horocycle}
The definitions below follow \cite[\S 8]{Gin89}. Consider an algebraically closed field $\Bbbk$ of characteristic zero and a connected reductive $\Bbbk$-group $G$. Fix a Borel subgroup $B \subset G$ with $U := R_u(B)$, and let $T := B/U$. Define
\begin{align*}
	Y & := U \backslash G, \quad \text{right $G$-action, left $T$-action}; \\
	Y^{\mathrm{op}} & := G/U, \quad \text{left $G$-action, right $T$-action}; \\
	\mathcal{B} & := T \backslash Y = B \backslash G, \quad \text{right $G$-action}; \\
	\mathcal{Y} & := Y^{\mathrm{op}} \utimes{T} Y, \quad \text{right $G^{\mathrm{op}} \times G$-action}.
\end{align*}
Here, the \emph{horocycle space} $\mathcal{Y}$ is formed by taking the quotient of the right $T$-action on $Y^{\mathrm{op}} \times Y$ via $(y, y') t = (yt, t^{-1} y')$. Observe that $\mathcal{Y}$ carries the free $T$-action $[y,y'] \cdot t = [yt, y'] = [y, ty']$ with quotient $\simeq \mathcal{B} \times \mathcal{B}$.

Consider the morphisms
\[\begin{tikzcd}[row sep=tiny]
	G & \mathcal{B} \times G \arrow[l, "p"'] \arrow[r, "q"] & \mathcal{Y} \\
	g & (Ty, g) \arrow[mapsto, l] \arrow[mapsto, r] & {[y^{-1}, yg]} .
\end{tikzcd}\]

Let $G^{\mathrm{op}} \times G$ act on the right of $\mathcal{B} \times G$ (resp.\ of $G$) by $(\beta, g) \cdot (a, b) = (\beta a^{-1}, agb)$ for all $a, b, g \in G$ and $\beta \in \mathcal{B}$ (resp.\ by bilateral translation).

\begin{lemma}\label{prop:horocycle-fiber}
	The morphisms $p, q$ are both $G^{\mathrm{op}} \times G$-equivariant. Moreover, $q$ is smooth affine and surjective, and for all $g_1, g_2 \in G$ we have 
	\[ q^{-1}\left( [g_1 U, U g_2] \right) = \{B g_1^{-1}\} \times g_1 U g_2 . \]
	In particular, $q^{-1}([g_1 U, g_2])$ is naturally a left $g_1 U g_1^{-1}$-torsor.
\end{lemma}
\begin{proof}
	The following diagram is Cartesian
	\[\begin{tikzcd}
		Y \times G \arrow[r, "\beta"] \arrow[d, "\alpha"'] & \mathcal{B} \times G \arrow[d, "q"] \\
		Y \times Y \arrow[r, "\gamma"'] & \mathcal{Y}
	\end{tikzcd}\]
	where $\alpha(y, g) = (y, y g)$, $\beta(y, g) = (T y, g)$ and $\gamma(y_1, y_2) = [y_1^{-1}, y_2]$. Therefore, by descent along $T$-torsors, it suffices to show $\alpha$ is smooth affine surjective. Smoothness and surjectivity are straightforward. To show $\alpha$ is affine, we use another Cartesian diagram:
	\[\begin{tikzcd}
		\left\{ (g, h_1, h_2) \in G^3: h_1 g h_2^{-1} \in U \right\} \arrow[r] \arrow[d, "{\alpha'}"'] & Y \times G \arrow[d, "\alpha"] \\
		G \times G \arrow[r] & Y \times Y
	\end{tikzcd} \quad \begin{tikzcd}
		(g, h_1, h_2) \arrow[mapsto, r] \arrow[mapsto, d] & (U h_1, g) \\
		(h_1, h_2) \arrow[mapsto, r] & (Uh_1, Uh_2)
	\end{tikzcd}\]
	The upper-left corner is closed in $G^3$, hence $\alpha'$ is affine. This property descends to $\alpha$ along $G \times G \twoheadrightarrow Y \times Y$.
	
	The description of $q^{-1}([g_1 U, U g_2])$ follows from $\alpha^{-1}(U g_1^{-1}, U g_2) = \{U g_1^{-1} \} \times g_1 U g_2$.
\end{proof}

The equivariance of $p, q$ justifies the following

\begin{definition}\label{def:HC-CH}
	Let $L \subset G^{\mathrm{op}} \times G$ be any subgroup and $\chi: \mathfrak{l} \to \Bbbk$ be a character. In the setting above, we set
	\[ \mathrm{HC}_{L, \chi} := q_* p^!, \quad \mathrm{CH}_{L, \chi} := p_! q^* \]
	in the equivariant derived categories of $(L, \chi)$-monodromic $\mathscr{D}$-modules (Definition \ref{def:equivariant-derived}). They give rise to a pair of adjoint functors
	\[\begin{tikzcd}
		\mathrm{CH}_{L, \chi}:  \arrow[r, yshift=0.3em] \cate{D}^{\mathrm{b}}_{L, \chi, \mathrm{h}}(\mathcal{Y}) & \cate{D}^{\mathrm{b}}_{L, \chi, \mathrm{h}}(G): \mathrm{HC}_{L, \chi} \arrow[l, yshift=-0.3em]
	\end{tikzcd}. \]
	Here $\cate{D}^{\mathrm{b}}_{L, \chi, \mathrm{h}}$ denotes the full triangulated subcategory of $\cate{D}^{\mathrm{b}}_{L, \chi}$ formed by complexes with holonomic cohomologies.
	
	When $L = \{1\}$, we denote them simply as $\mathrm{CH}$, $\mathrm{HC}$.
\end{definition}

For $L = \{1\}$ we obtain the non-equivariant version $\cate{D}^{\mathrm{b}}_{\mathrm{h}}(\mathcal{Y}) \leftrightharpoons \cate{D}^{\mathrm{b}}_{\mathrm{h}}(G)$, whilst for $\chi$ trivial we obtain the pair $\cate{D}^{\mathrm{b}}_{L, \mathrm{h}}(\mathcal{Y}) \leftrightharpoons \cate{D}^{\mathrm{b}}_{L, \mathrm{h}}(G)$ for complexes of equivariant $\mathscr{D}$-modules. Finally, these functors are also compatible with forgetful functors $\cate{D}^{\mathrm{b}}_{L, \chi, \mathrm{h}}(\cdots) \to \cate{D}^{\mathrm{b}}_{L', \chi', \mathrm{h}}(\cdots)$ when $L' \subset L$ is a subgroup and $\chi' = \chi|_{\mathfrak{l}'}$. Cf.\ the discussions after Definition \ref{def:equivariant-derived}.

\begin{theorem}[See {\cite[Theorem 8.5.1]{Gin89}}]\label{prop:summand}
	The identity functor of $\cate{D}^{\mathrm{b}}_{\mathrm{h}}(G)$ is a direct summand of $\mathrm{CH} \circ \mathrm{HC}$; this splits the adjunction co-unit $\mathrm{CH} \circ \mathrm{HC} \to \identity$.
\end{theorem}

In fact, it is shown in \textit{loc.\ cit.} that in the setting of constructible sheaves, $\mathrm{CH} \circ \mathrm{HC}$ is given by convolution with the Springer sheaf $\mathrm{Spr} \in \cate{Perv}(G)$, that is, $\Rder r_* \Bbbk_{\tilde{\mathcal{N}}} [\dim \tilde{\mathcal{N}}]$ where $r: \tilde{\mathcal{N}} \twoheadrightarrow \mathcal{N} \subset G$ is the Springer resolution. Moreover $\mathrm{Spr}$ is known to contain the skyscraper sheaf centered at $1$ as a direct summand, see \cite[8.9.17]{CG10}. These are transcribed to the $\mathscr{D}$-module setting in \cite[\S 8.7]{Gin89}.

\section{\texorpdfstring{$K$}{K}-admissible \texorpdfstring{$D$}{D}-modules: regularity}\label{sec:proof-regularity}
Retain the conventions from \S\ref{sec:horocycle}; in particular $\Bbbk$ is algebraically closed. Denote the $T$-torsor $\mathcal{Y} \to \mathcal{B} \times \mathcal{B}$ as $\pi$.

For the notion of regular holonomic $\mathscr{D}$-modules or complexes, we refer to standard references such as \cite[VII]{Bo87} or \cite[Chapter 6]{HTT08}.

\begin{lemma}\label{prop:regularity-prep}
	Let $L \subset G^{\mathrm{op}} \times G$ be a spherical subgroup, and $\chi$ be a reductive character of $\mathfrak{l}$. Suppose that $\mathcal{N}$ is a simple $(L, \chi)$-monodromic $\mathscr{D}_{\mathcal{Y}}$-module, $\mathcal{N}$ is holonomic, and that there exists a covering $\mathcal{B} \times \mathcal{B} = W_1 \cup \cdots \cup W_r$ by affine open subsets such that the sections of $\mathcal{N}|_{\pi^{-1} W_i}$ are all $\mathfrak{t}$-finite, for $i = 1, \ldots, r$. Then $\mathcal{N}$ is a regular holonomic $\mathscr{D}_{\mathcal{Y}}$-module.
\end{lemma}
\begin{proof}
	Apply the discussions around \eqref{eqn:categories} to the $T$-torsor $\pi$. In view of the $\mathfrak{t}$-finiteness assumption, one can decompose $\mathcal{N}$ into components indexed by $\overline{\xi} \in \mathfrak{t}^* / \mathfrak{t}^*_{\Z}$ according to \eqref{eqn:categories}. Since the actions of $L$ and $T$ commute, each component is still $(L, \chi)$-monodromic, and the simplicity implies that $\mathcal{N}$ is a simple object in $\mathscr{D}_{\tilde{X}} \dcate{Mod}_{\widetilde{\overline{\xi}}}$ for some $\overline{\xi}$; in fact, it must belong to $\mathscr{D}_{\tilde{X}} \dcate{Mod}_{\overline{\xi}}$.

	Pick a representative $\xi \in \mathfrak{t}^*$ of $\overline{\xi}$. By Proposition \ref{prop:finite-to-monodromic}, $\mathcal{N}$ acquires a canonical $(T, \xi)$-monodromic structure. Since $L$ and $T$ commute, $\mathcal{N}$ is actually $(L \times T, (\chi, \xi))$-monodromic. Noting that $L \times T$ acts on $\mathcal{Y}$ with finitely many orbits, one applies \cite[Lemma 2.5.1]{FG10} to conclude the regularity of $\mathcal{N}$. This is legitimate by the Remark below.
\end{proof}

\begin{remark}\label{rem:reductive-character}
	The result \cite[Lemma 2.5.1]{FG10} cited above asserts that if $Q$ is a connected reductive group, $\psi: \mathfrak{q} \to \Bbbk$ is a character, and $X$ is a smooth $Q$-variety with finitely many $Q$-orbits, then every $(Q, \psi)$-monodromic $\mathscr{D}_Q$-module $\mathcal{N}$ is regular holonomic. The case when $\psi$ is trivial and $Q$ is an arbitrary affine group is well-known; see \cite[Theorem 11.6.1]{HTT08}. What we need is the case $Q := L \times T$ acting on $X := \mathcal{Y}$ and $\psi := (\chi, \xi)$. This can be extracted from the proof of \cite[Lemma 2.5.1]{FG10} as follows. Choose a smooth $G$-equivariant compactification $j: \mathcal{Y} \hookrightarrow \overline{\mathcal{Y}}$. The goal is to show that $j_* \mathcal{N}$ is regular holonomic at every point. In \textit{loc.\ cit.}, this is reduced to the assertion that $\mathscr{O}_{Q, \psi}$ is regular holonomic as a $\mathscr{D}_{\overline{Q}}$-module, which is then established for connected reductive $Q$.

	To treat our case, first replace $L$ by $L^\circ$ to ensure $Q$ is connected. Take a Levi decomposition $Q = Q' \ltimes R_u(Q)$. Since $\psi$ is a reductive character, it decomposes into $\psi' \rtimes 0$. The regularity reduces to (a) the standard case of $\mathscr{O}_{R_u(Q)}$, and (b) the case of $\mathscr{O}_{Q', \psi'}$ which is addressed in \textit{loc.\ cit.}
	
	Note that reductivity is necessary in these arguments. If we take $Q = \Ga$ and $\psi$ nontrivial, then $\mathscr{O}_{Q, \psi} \simeq \mathscr{D}_Q / \mathscr{D}_Q \left( \frac{\dd}{\dd t} - a \right)$ for some $a \neq 0$. It is holonomic, yet irregular at $\infty$.
\end{remark}

\begin{definition}\label{def:K-admissible}
	Let $Z$ be a homogeneous $G$-variety. Let $K$ be a subgroup of $G$. A $D_Z$-module $M$ is called $K$-\emph{admissible} if
	\begin{itemize}
		\item $M$ is finitely generated over $D_Z$;
		\item $M$ is locally finite under $\mathcal{Z}(\mathfrak{g})$;
		\item the $\mathscr{D}_Z$-module $\mathcal{M}$ generated by $M$ is equipped with a $(K, \chi)$-monodromic structure for some reductive character $\chi$ of $\mathfrak{k}$.
	\end{itemize}
\end{definition}

Quotients and finitely generated submodules of a $K$-admissible $D_Z$-module are still admissible, provided that they are $K$-stable. By Lemma \ref{prop:monodromic-finiteness}, every $K$-admissible $D_Z$-module is $\mathfrak{k}$-admissible in the sense of Definition \ref{def:k-admissible}. For any $D_Z$-module $M$, we write $\mathcal{M} := \mathscr{D}_Z \cdot M$.

\begin{remark}
	One can view $G$ as a $G^{\mathrm{op}} \times G$-variety by $x \cdot (a, b) = axb$. The definition above can therefore be applied to $D_G$-modules under the action of a subgroup $L \subset G^{\mathrm{op}} \times G$. Note that the local finiteness under $\mathcal{Z}(\mathfrak{g} \times \mathfrak{g})$ and $\mathcal{Z}(\mathfrak{g})$ are equivalent.
\end{remark}

\begin{example}\label{eg:K-admissible}
	The examples mentioned in \S\ref{sec:monodromic} are directly related to $K$-admissibility. In what follows, we view $Z$ as a smooth variety over $\CC$ which is definable over $\R$.
	\begin{enumerate}[(i)]
		\item \emph{Function spaces} \quad In Example \ref{eg:function-space}, suppose furthermore that $V$ is finitely generated over $D_Z$ and locally finite under $\mathcal{Z}(\mathfrak{g})$, then $V$ generates a $K$-admissible $D_Z$-module with trivial $\chi$. Indeed, $\mathscr{D}_Z \cdot V$ is equipped with a $K$-equivariant structure.
		\item \emph{Relative invariants with reductive character $\chi$} \quad In Example \ref{eg:relative-invariant}, suppose that $u$ is $\mathcal{Z}(\mathfrak{g})$-finite, then $D_Z \cdot u$ is $K$-admissible; in this case, $\mathscr{D}_Z \cdot u$ is equipped with a $(K, \chi)$-monodromic structure.
		\item \emph{Localizations} \quad In Example \ref{eg:localization}, suppose that the $(\mathfrak{g}, K)$-module $V$ is a Harish-Chandra module; see \cite[\S 4]{BK14}. In this case $V$ is finitely generated over $U(\mathfrak{g})$ and locally finite under $\mathcal{Z}(\mathfrak{g})$. Hence $D_Z \dotimes{U(\mathfrak{g})} V$ is $K$-admissible with trivial $\chi$. The $\mathscr{D}_Z$-module it generates is $\Loc_Z(V)$ which is $K$-equivariant.
	\end{enumerate}
\end{example}

\begin{theorem}\label{prop:L-regularity}
	Let $L \subset G^{\mathrm{op}} \times G$ be a spherical subgroup. Then every $L$-admissible $D_G$-module $M$ is regular holonomic.
\end{theorem}
\begin{proof}
	By Corollary \ref{prop:holonomic} applied to $Z := G$ and $L$, we see $\mathcal{M}$ is holonomic. Theorem \ref{prop:summand} implies that $\mathcal{M}$ (as a $\mathscr{D}_G$-module) is a direct summand of $\mathrm{CH} \circ \mathrm{HC}(\mathcal{M})$. Since the functors $p_!$ and $q^*$ in derived categories preserve regular holonomic complexes (see \cite[Theorem 6.1.5]{HTT08}), the regularity of $\mathcal{M}$ will follow from that of the cohomologies of $\mathrm{HC}(\mathcal{M})$, which we prove below.
	
	Since
	\[ \mathrm{HC}(\mathcal{M}) := \mathrm{HC} \left(\mathbf{oblv}(\mathcal{M})\right) \simeq \mathbf{oblv} \left( \mathrm{HC}_{L, \chi}(\mathcal{M}) \right), \]
	the cohomologies $\mathrm{HC}^i(\mathcal{M})$ of $\mathrm{HC}(\mathcal{M})$ are endowed with $(L, \chi)$-monodromic structures (cf.\ Definition \ref{def:equivariant-derived} and the subsequent discussions), for any given $i \in \Z$. The $\mathscr{D}_{\mathcal{Y}}$-module $\mathrm{HC}^i(\mathcal{M})$ is holonomic, thus of finite length in the $(L, \chi)$-monodromic category. It suffices to show that each simple $(L, \chi)$-monodromic subquotient $\mathcal{N}$ of $\mathrm{HC}^i(\mathcal{M})$ is regular holonomic.
	
	In view of Lemma \ref{prop:regularity-prep}, it remains to check the $\mathfrak{t}$-local finiteness of $\mathcal{N}|_{\pi^{-1} W}$, where $W \subset \mathcal{B} \times \mathcal{B}$ ranges over some finite affine open covering; note that $\pi^{-1}W$ is still affine. This will follow from the same property for $\mathrm{HC}^i(\mathcal{M})$. As $p^! \mathcal{M} \simeq \mathscr{O}_{\mathcal{B}} \boxtimes \mathcal{M} [\dim \mathcal{B}]$ by \cite[VII.9.14 Corollary]{Bo87}, it remains to show the local $\mathfrak{t}$-finiteness of sections of $\Rder^j q_* (\mathscr{O}_{\mathcal{B}} \boxtimes \mathcal{M})$ over $\pi^{-1} W$, for all $j \geq 0$ and suitably chosen $W$.
	
	The required argument for the last step is given in \cite[p.156---158]{Gin93}. Let us conclude by a very brief sketch. Using the fact that $q$ is affine smooth and the description of its fibers (Lemma \ref{prop:horocycle-fiber}), one computes $\Rder^j q_* (\mathscr{O}_{\mathcal{B}} \boxtimes \mathcal{M})$ by an explicit relative de Rham resolution. The local $\mathfrak{t}$-finiteness is thus related to the known local $\mathcal{Z}(\mathfrak{g})$-finiteness of $M$ by the following observation. The $T$-action $[y_1, y_2] \mapsto [y_1, t y_2]$ on $\mathcal{Y}$ lifts to $\mathcal{B} \times G$ by
	\[ t: (Ty, g) \mapsto (Ty, y^{-1} t y g), \quad t \in T \]
	by which one computes the $\mathfrak{t}$-action on $\Rder^j q_* (\mathscr{O}_{\mathcal{B}} \boxtimes \mathcal{M})$. A standard fact says
	\begin{equation}\label{eqn:HC-unshifted}
		\mathcal{Z}(\mathfrak{g}) \subset U(\mathfrak{t}) \oplus U(\mathfrak{g}) \mathfrak{u}
	\end{equation}
	and the resulting projection $\mathcal{Z}(\mathfrak{g}) \to U(\mathfrak{t})$ so obtained is the Harish-Chandra map without shifting by half-sum of positive roots.
	
	Analogously, one may let $\mathfrak{u}$ act via $u: (Ty, G) \mapsto (Ty, y^{-1}uyg)$ by choosing local sections for $Y \to \mathcal{B}$. However $y^{-1} \mathfrak{u} y$ is the ``vertical direction'' over $q(Ty, g) = [y^{-1}, yg]$ relative to $q$ by Lemma \ref{prop:horocycle-fiber}. In view of \eqref{eqn:HC-unshifted}, this will eventually enable us to employ the local $\mathcal{Z}(\mathfrak{g})$-finiteness of $M$.
\end{proof}

\begin{corollary}[Cf.\ {\cite[Corollary 8.9.1]{Gin89}}]\label{prop:admissible-regular}
	Let $Z$ be a spherical homogeneous $G$-variety, and $K \subset G$ be a spherical subgroup. Then every $K$-admissible $D_Z$-module $M$ generates a regular holonomic $\mathscr{D}_Z$-module.
\end{corollary}
\begin{proof}
	We may assume $Z = H \backslash G$ where $H$ is a spherical subgroup of $G$. The quotient map $f: G \twoheadrightarrow Z$ is an $H$-torsor, hence smooth. By Corollary \ref{prop:holonomic}, $\mathcal{M}$ is holonomic, and so is $\mathcal{N} := f^* \mathcal{M}$. Note that $\mathcal{N}$ is concentrated at degree $0$ by the flatness of $f$, and it is generated by $N$, the finitely generated $D_G$-module formed by $f^*$-images of the elements of $M$.
	
	Let $L := H^{\mathrm{op}} \times K \subset G^{\mathrm{op}} \times G$. We contend that $N$ is $L$-admissible. Indeed, the local $\mathcal{Z}(\mathfrak{g})$-finiteness is inherited from $M$; so is the $H^{\mathrm{op}} \times K$-monodromic structure on $\mathcal{N}$ since it is pulled back from $H \backslash G$.
	
	Note that $L$ is a spherical subgroup of $G^{\mathrm{op}} \times G$. Therefore $\mathcal{N}$ is regular holonomic by Theorem \ref{prop:L-regularity}. This implies the regularity of $\mathcal{M}$ by \cite[VII. 12.9]{Bo87}.
\end{proof}

\section{Subanalytic sets and maps}\label{sec:subanalytic}
We will use the notion of subanalytic subsets and subanalytic functions on real analytic manifolds; the relevant theory can be found in \cite{BM88} or \cite[\S 8.2]{KS90}.

\begin{definition}\label{def:subanalytic-set}
	Let $M$ be a real analytic manifold. A subset $X \subset M$ is said to be semianalytic if each $x \in M$ has an open neighborhood $U$ such that $X \cap U = \bigcup_{i=1}^p \bigcap_{j=1}^q X_{ij}$, where each $X_{ij}$ is described by $f_{ij} = 0$ or $f_{ij} > 0$ for some family of analytic functions $f_{ij}: U \to \R$.
	
	We say $X \subset M$ is subanalytic if any $x \in X$ has an open neighborhood $U$ in $M$ such that $X \cap  U = \mathrm{pr}_1(A)$ for some relatively compact semianalytic subset $A \subset M \times N$, where $N$ is a real analytic manifold and $\mathrm{pr}_1: M \times N \to M$ is the projection.
\end{definition}

Below is a summary of basic properties we need. See the paragraph after \cite[Definition 3.1]{BM88},
\begin{itemize}
	\item Locally closed analytic submanifolds are semianalytic, hence subanalytic.
	\item Finite unions and finite intersections of subanalytic sets are subanalytic.
	\item Connected components of a subanalytic set are locally finite and subanalytic.
	\item The closure of a subanalytic subset is subanalytic.
	\item Complements of subanalytic sets are subanalytic; this is \cite[Theorem 3.10]{BM88}.
\end{itemize}

\begin{definition}
	Let $X \subset M$ be a subset, $N$ be a real analytic manifold. We say a function $f: X \to N$ is subanalytic if its graph $\Gamma_f \subset M \times N$ is subanalytic.
\end{definition}

\begin{itemize}
	\item Morphisms between analytic manifolds are subanalytic.
	\item The image of a relatively compact subanalytic set under a subanalytic mapping remains subanalytic; see the remark after \cite[Definition 3.2]{BM88}.
	\item Composites of subanalytic maps are subanalytic.
	\item Let $X \subset \R^n$ be a subanalytic subset, then the Euclidean distance $d(x, X)$ is subanalytic on $\R^n$; this is \cite[Remarks 3.11 (1)]{BM88}.
\end{itemize}

We are ready to state our main technical tool, \emph{Łojasiewicz's inequality}.

\begin{definition}\label{def:p-equivalent}
	Let $M$ be a set and $f, g: M \to \R_{\geq 0}$. We write $f \preccurlyeq g$ if there exist constants $a, C \in \R_{> 0}$ such that $f \leq C g^a$. If both $f \preccurlyeq g$ and $g \preccurlyeq f$ hold, we write $f \sim g$ and say they are \emph{power-equivalent}.
\end{definition}

\begin{theorem}[S.\ Łojasiewicz; see {\cite[Theorem 6.4]{BM88}}]\label{prop:inequality}
	Let $M$ be a real analytic manifold, $E \subset M$ a subanalytic subset and let $f, g: E \to \R$ be subanalytic functions with compact graphs in $E \times \R$. If $f^{-1}(0) \subset g^{-1}(0)$, then $|g| \preccurlyeq |f|$.
\end{theorem}

As a particular case, the assumptions hold when $E$ is compact subanalytic and $f, g: E \to \R_{\geq 0}$ are continuous subanalytic functions. In that case, $f^{-1}(0) = g^{-1}(0)$ if and only if $f$ and $g$ are power-equivalent. 

We record some easy observations for later use.
\begin{lemma}\label{prop:subanalytic-descent}
	Let $\pi: X \to Y$ be a locally trivial fibration between real analytic manifolds.
	\begin{enumerate}[(i)]
		\item Let $E \subset Y$ be a subset. If $\pi^{-1}(E)$ is subanalytic in $X$, then $E$ is subanalytic in $Y$.
		\item Let $f_Y: Y \to \R$ be a function such that $f_X := f_Y \circ \pi$ is subanalytic on $X$, then $f_Y$ is subanalytic on $Y$.
	\end{enumerate}
\end{lemma}
\begin{proof}
	Consider (i). By the local nature of Definition \ref{def:subanalytic-set}, upon retracting $Y$ we may assume $X = Y \times F$ and $\pi$ is the first projection, where $F$ is some real analytic manifold. For every $y \in Y$, pick $f \in F$. Since $E \times F$ is subanalytic in $Y \times F$,  there exist
	\begin{compactitem}
		\item an open neighborhood $U_Y \times U_F$ of $(y, f)$ in $Y \times F$,
		\item a real analytic manifold $N$,
		\item a relative compact semianalytic subset $A \subset (Y \times F) \times N$,
	\end{compactitem}
	such that $(E \times F) \cap (U_Y \times U_F) = \mathrm{pr}_{12}(A)$, where $\mathrm{pr}_{12}: (Y \times F) \times N \to Y \times F$ is the projection. It follows that $E \cap U_Y  = \pi \left( \mathrm{pr}_{12}(A) \right)$.

	Taking the ``$N$'' in Definition \ref{def:subanalytic-set} to be the $N \times F$ above, the preceding discussion shows that $E$ is subanalytic, by varying $y$.

	As for (ii), we have to show the graph $\Gamma_{f_Y} \subset Y \times \R$ is subanalytic. Observe that $\Gamma_{f_X} = (\pi \times \identity_{\R})^{-1}(\Gamma_{f_Y})$. We conclude by applying (i) to $\pi \times \identity_{\R}: X \times \R \to Y \times \R$.
\end{proof}

\section{Growth conditions}\label{sec:growth}
\begin{definition}\label{def:p-growth}
	Consider a set $M$, its subset $V$ and a function $p: M \to \R_{\geq 0}$. We say a function $f: V \to \CC$ has $p$-bounded growth relative to $M$, if there exists $a \in \R_{> 0}$ such that $p^a |f|$ is bounded on $V$. This notion depends only on the power-equivalence class of $p$ (Definition \ref{def:p-equivalent}).
\end{definition}

\begin{lemma}\label{prop:p-growth-pullback}
	Let $M$ be a topological space, $V$ an open subset and $p: M \to \R_{\geq 0}$ be continuous. Let $\pi: M' \to M$ be a continuous map. If $f: V \to \CC$ has $p$-bounded growth, then $f \pi: \pi^{-1}(V) \to \CC$ has $p\pi$-bounded growth; the converse holds if $f$ is continuous and $\pi: \pi^{-1}(V) \to V$ has dense image.
\end{lemma}
\begin{proof}
	Immediate.
\end{proof}

The utility of this notion is explained by the following result.
\begin{lemma}\label{prop:existence-p}
	Suppose that $M$ is a compact real analytic manifold, and $U \subset M$ is an open subanalytic subset. Then there exists a subanalytic continuous function $p: M \to \R_{\geq 0}$ such that $U = \left\{ x \in M: p(x) > 0 \right\}$.
\end{lemma}
\begin{proof}
	Without loss of generality we may assume $M$ connected. Recall that $M \smallsetminus U \subset M$ is closed and subanalytic. There exists a closed immersion $i: M \hookrightarrow \R^N$ of real analytic spaces, by \cite[Theorem 1]{ABT79}. Now take
	\[ p(x) := d\left( i(x), i(M \smallsetminus U)\right), \quad x \in M \]
	where $d$ is the Euclidean distance function on $\R^N$. Since $i(M \smallsetminus U) \subset \R^n$ is subanalytic, $d(\cdot, i(M \smallsetminus U))$ is subanalytic continuous on $\R^N$, hence so is $p$.
\end{proof}

Now we turn to the case of real algebraic varieties.
\begin{definition-proposition}\label{def:moderate}
	Let $X$ be a smooth $\R$-variety.
	\begin{itemize}
		\item There exists an open immersion $j: X \to \overline{X}$ with Zariski-dense image, such that $\overline{X}$ is smooth and $\overline{X}(\R)$ is compact.
		\item For each $j$ above, there exists a continuous subanalytic function $p: \overline{X}(\R) \to \R_{\geq 0}$ such that
		\[ j(X)(\R) = \left\{ x \in \overline{X}(\R) : p(x) > 0 \right\}. \]
		In this case, $p$ is said to be \emph{adapted} to $j$.
	\end{itemize}

	Let $V$ be a connected component of $X(\R)$. We say a continuous function $f: V \to \CC$ has \emph{moderate growth at infinity} if $f$ has $p$-bounded growth relative to $\overline{X}(\R)$. This notion is independent of $j: X \hookrightarrow \overline{X}$ and $p$.
\end{definition-proposition}
\begin{proof}
	The existence of $j$ is ensured by Nagata's theorem followed by Hironaka's resolution of singularities. The existence of $p$ follows from Lemma \ref{prop:existence-p}. Consider the category $\mathcal{C}$ of open immersions $j: X \to \overline{X}$ as above, the morphisms from $j_1: X \to \overline{X}_1$ to $j_2: X \to \overline{X}_2$ being morphisms $\pi: \overline{X}_1 \to \overline{X}_2$ such that $\pi j_1 = j_2$. We claim that $\mathcal{C}$ is co-filtrant.
	
	Indeed, given $j_i: X \to \overline{X}_i$ for $i=1, 2$, take $\overline{X}'$ to be the schematic closure of the diagonal image of $X$ in $\overline{X}_1 \dtimes{\R} \overline{X}_2$. Then we obtain an open dense immersion $j' = j_1 \times j_2: X \hookrightarrow \overline{X}'$; thus $\overline{X}'(\R)$ is compact, but $\overline{X}'$ is not necessarily smooth. To remedy this, take a resolution of singularities $\rho: \overline{X} \twoheadrightarrow \overline{X}'$ which is proper and restricts to $\rho^{-1} (j'(X)) \rightiso j'(X)$. Then $j: X \hookrightarrow \overline{X}$ dominates both $j_1$ and $j_2$ in $\mathcal{C}$.

	Let $f: V \to \CC$ be a continuous function. Extending $f$ by zero, we may assume $V = X(\R)$. Consider a morphism $\pi$ in $\mathcal{C}$ from $j': X \to \overline{X}'$ to $j: X \to \overline{X}$. Let $p$ (resp.\ $p'$) be adapted to $j$ (resp.\ $j'$). We claim that $f$ has $p$-bounded growth relative to $\overline{X}(\R)$ if and only if it has $p'$-bounded growth relative to $\overline{X}'(\R)$. In view of the previous step, this will entail that the notion of moderate growth at infinity is independent of all choices.
	
	We first show that $\pi^{-1}(j(X)) = j'(X)$. This is because $\pi: \pi^{-1}(j(X)) \to j(X)$ has a section $\sigma: j(X) \leftiso X \rightiso j'(X) \subset \pi^{-1}(j(X))$, thus $\sigma$ is a closed immersion by \cite[28.3.1]{stacks-project} since $\pi^{-1}(j(X))$ is separated. As $j'(X)$ is also open in the irreducible subset $\pi^{-1}(j(X))$, we see $\pi^{-1}(j(X)) = j'(X)$. Now $p', p \pi: \overline{X}'(\R) \to \R_{\geq 0}$ both have $j'(X)(\R)$ as their non-zero locus. Theorem \ref{prop:inequality} implies that $p'$ and $p\pi$ are power-equivalent. Hence $p'$-bounded growth and $p\pi$-bounded growth relative to $\overline{X}'(\R)$ are equivalent.

	Moreover, $p\pi$-bounded growth relative to $\overline{X}'(\R)$ is equivalent to $p$-bounded growth relative to $\overline{X}(\R)$ for continuous functions on $X(\R)$ (Lemma \ref{prop:p-growth-pullback}). This establishes our claim.
\end{proof}

For a systematic treatise of tempered functions on manifolds definable in polynomially bounded $o$-minimal structures, see \cite{Sha20}. The author is grateful to one of the referees for this suggestion.

Below is an intrinsic characterization of the functions $p|_{X(\R)}$, or rather their inverses.
\begin{proposition}\label{prop:w}
	Let $j: X \hookrightarrow \overline{X}$ be as in Definition--Proposition \ref{def:moderate}. Suppose that $w: X(\R) \to \R_{> 0}$ satisfies
	\begin{compactenum}[(i)]
		\item $w$ is continuous and subanalytic;
		\item for each constant $B > 0$, the subset $\{ x \in X(\R): w(x) \leq B \}$ is compact.
	\end{compactenum}
	Then $w^{-1}$ extends continuously to a unique subanalytic function $p: \overline{X}(\R) \to \R_{\geq 0}$ adapted to $j$. Conversely, for any $p$ adapted to $j$, the function $w(x) := p(x)^{-1}$ on $X(\R)$ satisfies (i) and (ii). 
\end{proposition}
\begin{proof}
	Put $\partial X := \overline{X} \smallsetminus X$. We have to extend $w^{-1}: X(\R) \to \R_{> 0}$ to a continuous subanalytic function on $\overline{X}(\R)$ adapted to $j$. By (ii), we see $w(x) \to +\infty$ when $x$ tends to $\partial X(\R)$. Let $\Gamma_{w^{-1}} \subset X(\R) \times \R_{> 0}$ be the graph of $w^{-1}$, which is subanalytic by (i). Its closure $\overline{\Gamma_{w^{-1}}}$ in $\overline{X}(\R) \times \R$ is still subanalytic; moreover $\overline{\Gamma_{w^{-1}}} \cap (\partial X(\R) \times \R) = \partial X(\R) \times \{0\}$. Hence $w^{-1}$ extends to a continuous subanalytic function $p: \overline{X}(\R) \to \R$ by zero. The converse is easy.
\end{proof}

\begin{remark}\label{rem:w}
	The real algebraic structures of $X$ and $\overline{X}$ play no roles in the proof above. Furthermore, we do not need to assume $\overline{X}(\R)$ is smooth: it suffices to embed it into a real analytic manifold in order to talk about subanalyticity.
\end{remark}

When $X$ is a homogeneous $G$-variety for a connected reductive $\R$-group $G$, we shall choose $w$ with additional properties. To begin with, define the norm $\|\cdot\|: G(\R) \to \R_{> 0}$ as in \cite[2.1.2]{BK14}. More precisely, choose an algebraic embedding $\iota: G \to \GL(N)$ and set
\begin{equation}\label{eqn:norm-group}
	\|g\| := \Tr\left( \iota(g) \cdot {}^t \iota(g) \right) + \Tr\left( \iota(g^{-1}) \cdot {}^t \iota(g)^{-1} \right), \quad g \in G(\R).
\end{equation}
On the other hand, we also have the function on the connected component $G(\R)^\circ$ defined by $\|g\|_{\max} := \exp(d(g, 1))$ where $d$ comes from a left-invariant Riemannian metric on $G(\R)^\circ$. According to \cite[Lemma 2.1]{BK14}, $\|\cdot\|_{\max}$ and $\|\cdot\|$ are power-equivalent as functions on $G(\R)^\circ$.

The following is a variant of the \emph{weight functions} discussed in \cite[5.3]{KSS14}, which is suitable for harmonic analysis on $G$-varieties.

\begin{lemma}\label{prop:weight}
	Suppose that $X$ is a homogeneous $G$-variety. There exists a function $w: X(\R) \to \R_{\geq 1}$ such that the properties (i) and (ii) in Proposition \ref{prop:w} are satisfied, and that there exists $C > 0$ and $N \in \Z_{\geq 1}$ such that $w(xg) \leq C \|g\|^N w(x)$ for all $g \in G(\R)^\circ$ and $x \in X(\R)$.
\end{lemma}
\begin{proof}
	Choose points $x_1, \ldots, x_n \in X(\R)$ so that $X(\R)$ decomposes into connected components
	\[ X(\R) = \bigsqcup_{i=1}^n x_i G(\R)^\circ \simeq \bigsqcup_{i=1}^n H_i \backslash G(\R)^\circ , \quad H_i := \Stab_{G(\R)^\circ}(x_i). \]
	Set $X_i := H_i \backslash G(\R)^\circ$. It suffices to fix $1 \leq i \leq n$ and define a function $w: X_i \to \R_{> 0}$ with the required properties.
	
	Consider $\hat{w}(x_i g) := \exp(d(H_i, g))$. It is continuous and subanalytic in $g \in G(\R)^\circ$ since $d(H_i, g)$ is. Indeed, $d(\cdot, \cdot)$ is subanalytic on $G(\R)^\circ \times G(\R)^\circ$ by a general result \cite[Theorem 3.5.2]{Ta81}; as for $d(H_i, \cdot)$, repeat the arguments in \cite[Remarks 3.11]{BM88}.

	The function $\hat{w}$ factors through $w: X_i \to \R_{> 0}$. Lemma \ref{prop:subanalytic-descent} (ii) implies $w$ is subanalytic, and $w$ is clearly continuous. Recall that $d(\cdot, \cdot)$ is left invariant. For any $B$, the closed subset $\{ x \in X_i: w(x) \leq B \}$ is compact since it is contained in the image of the compact $\{ g \in G(\R)^\circ : \|g\|_{\max} \leq 2B \}$ under $g \mapsto x_i g$.
	
	Suppose $g, t \in G(\R)^\circ$. From $d(H_i, gt) \leq d(H_i, g) + d(g, gt)$ we obtain $\hat{w}(x_i gt) \leq \|t\|_{\max} \hat{w}(x_i g)$. The required estimate on $w(xg)$ follows from the power-equivalence between $\|\cdot\|$ and $\|\cdot\|_{\max}$.
\end{proof}

Observe that if $w$ satisfies the requirements of Lemma \ref{prop:weight}, so does $w^\alpha$ for any $\alpha > 0$.

\section{Growth of regular holonomic solutions}\label{sec:Deligne}
We apply the formalism of \S\ref{sec:growth} to the solutions of regular holonomic systems. As the first step, we relate the notion of $p$-bounded growth to the following growth condition taken from \cite{Del70,KS16} that appears frequently in microlocal analysis.

\begin{definition}\label{def:growth}
	Let $M$ be a real analytic manifold and $V \subset M$ be an open subset. We say a continuous function $f$ on $V$ has \emph{polynomial growth} at $x \in M$ if for any sufficiently small compact neighborhood $K \ni x$ in $M$, there exists $N \in \Z_{\geq 1}$ such that
	\begin{equation}\label{eqn:tempered}
		\sup_{y \in K \cap V} d(y, K \smallsetminus V)^N |f(y)| < +\infty;
	\end{equation}
	here $d$ is the Euclidean distance relative to an analytic coordinate chart on $K$, and the $\sup := 0$ when $K \cap V = \emptyset$ or $K \subset V$. We say $f$ has polynomial growth relative to $M$ if it so at every $x$.
\end{definition}

It follows from Łojasiewicz's inequality (Theorem \ref{prop:inequality}, and also \cite[Remark 6.5]{BM88}) that the foregoing definition is independent of local coordinate charts. Besides, only the behavior of $f$ around the boundary $\partial V$ matters. Its relation to $p$-bounded growth is explicated as follows.

\begin{proposition}\label{prop:polynomial-to-p}
	Let $V$ be a subanalytic open subset of a compact real analytic manifold $M$, and $p: M \to \R_{\geq 0}$ be a continuous subanalytic function. Suppose that $V \subset \left\{x \in M: p(x) > 0 \right\}$. If a continuous function $f: V \to \CC$ has polynomial growth relative to $M$, then $f$ has $p$-bounded growth.
\end{proposition}
\begin{proof}
	Cover $\partial V$ by finitely many compact neighborhoods $K_1, \ldots, K_m$ as above in Definition \ref{def:growth}; for each $i$ we have chosen an analytic coordinate chart $K_i \hookrightarrow \R^n$ where $n = \dim M$, with the corresponding distance function $d_i$ and the exponent $N_i$ in \eqref{eqn:tempered}; we may also assume $d_i(y, K_i \smallsetminus V) \leq 1$ and $p(y) \leq 1$ for all $y \in K_i \cap V$.
	
	Since $p(y) = 0 \implies y \notin V \implies d_i(y, K_i \smallsetminus V) = 0$ for all $y \in K_i$, Theorem \ref{prop:inequality} (see also \cite[Remark 6.5]{BM88}) then implies $p(y)^{r_i} \geq c_i d_i(y, K_i)^{N_i}$ for some constants $c_i, r_i > 0$, for all $1 \leq i \leq m$ and $y \in K_i \cap V$. Taking $r := \max\{r_1, \ldots, r_m\} $, we see $p^r |f|$ is bounded on $V$.
\end{proof}

\begin{corollary}
	Let $X$ be a smooth $\R$-variety and $V$ be a connected component of $X(\R)$. Every continuous function $f: V \to \CC$ of polynomial growth automatically has moderate growth at infinity.
\end{corollary}

Suppose $U$ is a smooth $\R$-variety, and let $\mathcal{M}$ be a $\mathscr{D}_{U_{\CC}}$-module generated by some global section $\mu$. Let $V$ be an open subset of $U(\R)$ and $u: V \to \CC$ be an analytic function. Therefore $u$ extends holomorphically to an open subset $\mathcal{V} \subset U^{\ana}$ containing $V$. We say $u$ is an \emph{analytic solution} to $\mathcal{M}$, if $\mu \mapsto u$ induces a homomorphism of $\mathscr{D}_{\mathcal{V}}$-modules for some $\mathcal{V}$ as above; here we also employ the language of analytic $\mathscr{D}$-modules on complex manifolds.

\begin{theorem}\label{prop:growth-estimate}
	Let $U$ be a smooth $\R$-variety and $\mathcal{M}$ be a regular holonomic $\mathscr{D}_{U_{\CC}}$-module generated by some $\mu \in \Gamma(U, \mathcal{M})$. Let $V$ be a connected component of $U(\R)$. Then every analytic solution $u: V \to \CC$ to $\mathcal{M}$ has moderate growth at infinity in the sense of Definition--Proposition \ref{def:moderate}.
\end{theorem}
\begin{proof}
	Since $\mathcal{M}$ is holonomic, there exists an open $U_0 \subset U$ such that $\mathcal{M}$ is an integrable connection on $U_0$. Our aim is to show that $u$ is of $p$-bounded growth relative to $X(\R)$, for any data $(X, U, U_0, V, \mathcal{M}, u, p)$ where
	\begin{compactitem}
		\item $X$ is a smooth proper $\R$-variety, together with an open dense immersion $U \hookrightarrow X$;
		\item $U_0 \subset U$ is open dense;
		\item $V$ is a connected component of $U(\R)$;
		\item $\mathcal{M}$ is a regular holonomic $\mathscr{D}_U$-module, generated by some global section $\mu$ and $\mathcal{M}|_{U_0}$ is an integrable connection;
		\item $u$ is an analytic solution to $\mathcal{M}$ on $V$;
		\item $p: X(\R) \to \R_{\geq 0}$ is adapted to $U \hookrightarrow X$ in the sense of Definition--Proposition \ref{def:moderate}.
	\end{compactitem}
	Here we require $X$ to be a proper $\R$-scheme, which is stronger than the compactness of $X(\R)$.
	
	Consider a proper surjective morphism $\pi: X' \to X$ between $\R$-varieties. Set
	\[ U' := \alpha^{-1}(U), \quad U'_0 := \alpha^{-1}(U_0), \quad u' := u \circ \pi, \quad p' := p \circ \pi. \]
	Let $\mathcal{M}'$ be the $\mathscr{D}_{U'}$-module $\Lder^0 \pi^* \mathcal{M}$. It is still regular holonomic, generated by $\mu' := 1 \otimes \mu$, and is an integrable connection on $U'_0$; then $u'$ is an analytic solution to $\mathcal{M'}|_{U'}$. To estimate $u'$, we restrict it to a connected component $V'$ of $\alpha^{-1}(V)$. Lemma \ref{prop:p-growth-pullback} and Definition--Proposition \ref{def:moderate} entail that the case for $(X', U', U'_0, V', \mathcal{M}', u', p')$, for various connected components $V'$, will imply the case for $(X, U, U_0, V, \mathcal{M}, u, p)$. Some preliminary reductions are in order.
	
	\begin{enumerate}
		\item First, we reduce to the case where $X \smallsetminus U_0$ and its closed subset $X \smallsetminus U$ are both divisors. This is easily achieved by blowing up.
		\item Next, we take $\pi: X' \to X$ so that $\pi^{-1}(X \smallsetminus U_0)$ is a divisor with normal crossings. This can be done by Hironaka's theorem, but de Jong's alteration \cite[Theorem 4.1]{dJ96} suffices for our purpose as $X$ is proper. Then $\pi^{-1}(X \smallsetminus U)$ is also a divisor, as any preimage of a divisor does. 
	\end{enumerate}
	
	Now study the behavior of $u$ around some $x \in \partial V$ in $X(\R)$. Let $\mathcal{D}$ denote the unit open disc in $\CC$. We may choose local coordinates $z_1, \ldots, z_n$ on an open neighborhood $O \ni x$ in $X(\CC)$, such that
	\begin{gather*}
	(z_1, \ldots, z_n): O \rightiso \mathcal{D}^n, \quad x \mapsto (0, \ldots, 0) \\
	O \cap (X \smallsetminus U_0) = \left\{ z_1 \cdots z_n = 0 \right\}, \quad O \cap (X \smallsetminus U) = \left\{ z_1 \cdots z_a = 0 \right\},
	\end{gather*}
	for some $0 \leq a \leq n$. Therefore $O \cap U = \{ z_1 \cdots z_a \neq 0 \}$ and $O \cap V$ is a union of connected components of $O \cap U$.
	
	The section $u|_{O \cap V \cap U_0(\R)}$ of the local system associated with $\mathcal{M}|_{U_0}$ extends to a multi-valued section on $O \cap (X \smallsetminus U_0)(\CC)$, i.e.\ a section on the universal covering. It is a well-known virtue of regular holonomic systems (see eg.\ \cite[III.1]{Del70}, \cite[IX.2.2]{Ph11}) that the analytically continued $u$ can be expressed as a finite sum
	\begin{equation}\label{eqn:Nilsson}
		u = \sum_{\mathbf{s}, \mathbf{m}} \Phi_{\mathbf{s}, \mathbf{m}}(z) z^{\mathbf{s}} \log^{\mathbf{m}}(z),
	\end{equation}
	with
	\begin{gather*}
		\mathbf{s} = (s_1, \ldots, s_n) \in \CC^n, \quad \mathbf{m} = (m_1, \ldots, m_n) \in \Z_{\geq 0}^n, \\
		z^{\mathbf{s}} := \prod_{i=1}^n z_i^{s_i}, \quad \log^{\mathbf{m}}(z) := \prod_{i=1}^n (\log z_i)^{m_i}, \\
	\Phi_{\mathbf{s}, \mathbf{m}}: \;\text{holomorphic functions on } \mathcal{D}^n,
	\end{gather*}
	where we take the usual branches of $\log$ and $z^s$. To ensure uniqueness, we may assume $\mathbf{s}$ ranges over representatives of $\CC^n/\Z^n$ (see the Remark after the cited result in \cite{Ph11}). The standard generators $g_1, \ldots, g_n$ of $\pi_1(O \cap (X \smallsetminus U_0)(\CC), x) \simeq \Z^n$ act as
	\begin{equation}\label{eqn:monodromy}
		z_i^{s_i} \xmapsto{g_i} \exp\left(2\pi \sqrt{-1} s_i \right) z_i^{s_i}, \quad \log z_i \xmapsto{g_i} \log z_i + 2\pi \sqrt{-1}, \quad i = 1, \ldots, n.
	\end{equation}
	
	The $g_i$-action on $u$ for $a < i \leq n$ is realized by analytic continuation along the loop $z_i = \epsilon \exp(2\pi \sqrt{-1}\theta)$ where $0 < \epsilon \ll 1$ and $\theta \in [0, 2\pi]$; the other coordinates $z_j$ are nonzero constants. But $u$ is analytic on $V$, hence holomorphic in some open neighborhood of $(z_1, \ldots, \underbracket{0}_{i}, \ldots, z_n) \in O \cap V$ inside $X(\CC)$. The monodromic action $g_i$ is thus trivial on $u$ when $a < i \leq n$.
	
	By comparison with \eqref{eqn:monodromy}, we conclude that \eqref{eqn:Nilsson} involves only terms with
	\[ \mathbf{s} = (s_1, \ldots, s_a, 0, \ldots, 0), \quad \mathbf{m} = (m_1, \ldots, m_a, 0, \ldots, 0). \]
	Therefore $u|_{V \cap O}$ has polynomial growth relative to $X(\R)$ by \eqref{eqn:Nilsson}; multi-valuedness is not an issue since $V \cap O$ has contractible connected components. Apply Proposition \ref{prop:polynomial-to-p} to deduce $p$-bounded growth.
\end{proof}

\begin{remark}
	The case $U = U_0$ of Theorem \ref{prop:growth-estimate} (see the proof) is recorded in \cite[Théorème II.4.1]{Del70}.
\end{remark}

\section{Applications to admissible distributions}\label{sec:app}
Throughout this section, the connected reductive group $G$, its subgroups and homogeneous spaces are all defined over $\R$, but the $\mathscr{D}$-modules will live over $\CC$. It is thus convenient to adopt the classical viewpoint that the groups and varieties are over $\CC$, but also carry $\R$-structures. In particular, the $\mathscr{D}$-modules in question live on $\CC$-varieties. We write $\mathscr{D}_{Z^{\ana}}$ for the sheaf of analytic differential operators on $Z^{\ana}$, the $\CC$-analytic variety associated with $Z$.
% We write $\mathscr{D}_Z$ instead of $\mathscr{D}_{Z^{\ana}}$, etc.

For any smooth variety $Z$ defined over $\R$, we view $Z(\R)$ as a real analytic manifold. For a $D_Z$-module $M$, we denote the $\mathscr{D}_Z$-module it generates as $\mathcal{M}$ as usual. Hereafter, $Z$ will be a homogeneous $G$-variety and $K \subset G$ will be a subgroup.

\begin{definition}\label{def:admissible-distribution}
	Let $u$ be distribution, or more generally a hyperfunction on $Z(\R)$ in Sato's sense. We say $u$ is $K$-admissible (resp.\ $\mathfrak{k}$-admissible) if there exist
	\begin{compactitem}
		\item a $K$-admissible (resp.\ $\mathfrak{k}$-admissible) $D_Z$-module $M$,
		\item a subquotient $\mathcal{N}$ of $\mathcal{M}$ in $\mathscr{D}_Z \dcate{Mod}$ and an isomorphism $\mathscr{D}_Z \cdot u \simeq \mathcal{N}$.
	\end{compactitem}
\end{definition}

When $Z$ is spherical and $K$ is a spherical subgroup, Corollary \ref{prop:admissible-regular} implies that every $K$-admissible hyperfunction generates a regular holonomic $\mathscr{D}_Z$-module, and Corollary \ref{prop:holonomic} implies that every $\mathfrak{k}$-admissible hyperfunction generates a holonomic $\mathscr{D}_Z$-module.

We shall study the $K$-admissible hyperfunctions through the solution complexes of $\mathscr{D}_Z$-modules. Following \cite[XI]{KS90}, let
\[ \mathscr{B}_{Z(\R)} \supset \mathfrak{Db}_{Z(\R)} \supset \mathscr{C}^\infty_{Z(\R)} \supset \mathscr{A}_{Z(\R)} \]
denote the sheaves of hyperfunctions, distributions, $C^\infty$-functions, and analytic functions on $Z(\R)$, respectively. Extending by zero, they are also viewed as sheaves on $Z^{\ana}$; in fact they are $\mathscr{D}_{Z^{\ana}}$-modules. Note that
\[ \mathscr{A}_{Z(\R)} = \mathscr{O}_{Z^{\ana}}|_{Z(\R)}, \quad  \mathscr{B}_{Z(\R)} = \mathrm{R}^{\dim Z}\Gamma_{Z(\R)}\left(\mathscr{O}_{Z^{\ana}}\right) \otimes \mathrm{or}_{Z(\R)}, \]
where $\mathrm{or}_{Z(\R)}$ is the orientation sheaf.

Hereafter, assume $Z$ is a spherical homogeneous $G$-variety and $K \subset G$ is a spherical subgroup. For every regular holonomic $\mathscr{D}_Z$-module $\mathcal{M}$, we obtain from \cite[Corollary 8.3 and 8.5]{Ka84} the quasi-isomorphisms
\begin{gather*}
	\iHom_{\mathscr{D}_{Z^{\ana}}} \left(\mathcal{M}^{\ana}, \mathfrak{Db}_{Z(\R)}\right) \rightiso \iHom_{\mathscr{D}_{Z^{\ana}}} \left(\mathcal{M}^{\ana}, \mathscr{B}_{Z(\R)}\right), \\
	\iHom_{\mathscr{D}_{Z^{\ana}}} \left(\mathcal{M}^{\ana}, \mathscr{A}_{Z(\R)}\right) \rightiso \iHom_{\mathscr{D}_{Z^{\ana}}} \left(\mathcal{M}^{\ana}, \mathscr{C}^\infty_{Z(\R)}\right).
\end{gather*}
These $\iHom_{\mathscr{D}_{Z^{\ana}}}(\mathcal{M}^{\ana}, \cdot)$ are the \emph{solution complexes} of $\mathcal{M}$ valued in various function spaces.

\begin{theorem}\label{prop:hyperfcn-dist}
	Assume $Z$ is spherical homogeneous and $K \subset G$ is a spherical subgroup. Every $K$-admissible hyperfunction on $Z(\R)$ is a distribution, and every $K$-admissible $C^\infty$-function on $Z(\R)$ is analytic.
\end{theorem}
\begin{proof}
	Consider a hyperfunction $u$ on $Z(\R)$, a $K$-admissible $D_Z$-module $M$ and a subquotient $\mathcal{N}$ of $\mathcal{M}$ such that $\mathscr{D}_Z \cdot u \simeq \mathcal{N}$. Since $\mathcal{M}$ is regular holonomic by Corollary \ref{prop:admissible-regular}, so is $\mathcal{N}$. On the other hand $u$ can be identified as an element of $\Hom_{\mathscr{D}_{Z^{\ana}}} \left(\mathcal{N}^{\ana}, \mathscr{B}_{Z(\R)}\right)$. The same holds for distributions, $C^\infty$-functions and analytic functions. It remains to take $\mathrm{H}^0$ in the quasi-isomorphisms above.
\end{proof}

In the next two examples, the group acting on homogeneous spaces is always $G^{\mathrm{op}} \times G$, and the action is written as $\gamma (a,b) = a\gamma b$.

\begin{example}[Twisted characters]\label{eg:HC-character}
	Take $\tilde{G}$ to be a twisted space under $G$ (Example \ref{eg:group-case-U} with $\Bbbk = \R$) and take $K = \{ (g^{-1}, g): g \in G \}$.  We also fix a smooth character $\omega: G(\R) \to \CC^\times$ and consider the distributions $\Theta$ on $\tilde{G}(\R)$ satisfying
	\begin{equation}\label{eqn:omega-character}
		\Theta({}^g f) = \omega(g^{-1}) \Theta(f), \quad {}^g f(\gamma) = f(g^{-1}\gamma g)
	\end{equation}
	for all $g \in G(\R)$. A typical source of such distributions on $\tilde{G}(\R)$ is the \emph{$\omega$-twisted character}. We follow \cite[I.2.6]{HL17} to define them. First, we define a smooth $\omega$-representation to be a pair $(\pi, \tilde{\pi})$ where $\pi$ is an SAF representation of $G(\R)$ (see \cite[p.46]{BK14}) with underlying Fréchet space $V_\pi$, and $\tilde{\pi}: \tilde{G}(\R) \to \Aut_{\CC}(V_\pi)$ is such that
	\begin{compactitem}
		\item $\tilde{\pi}(a\gamma b) = \pi(a) \tilde{\pi}(\gamma)\pi(b) \cdot \omega(b)$ for all $a, b \in G(\R)$ and $\gamma \in \tilde{G}(\R)$,
		\item for some $\gamma$ (equivalently, for any $\gamma$) in $\tilde{G}(\R)$, the endomorphism $\tilde{\pi}(\gamma)$ of $V_\pi$ is invertible and continuous.
	\end{compactitem}

	For every $f \in C^\infty_c(\tilde{G}(\R))$, set
	\[ \tilde{\pi}(f) := \int_{\tilde{G}(\R)} f(\gamma) \tilde{\pi}(\gamma) \dd_\ell \gamma \; \in \End_{\CC}(V_\pi) \]
	by fixing a left $G(\R)$-invariant measure $\dd_\ell \gamma$ on $\tilde{G}(\R)$. If we fix $\gamma_0 \in \tilde{G}(\R)$ and set
	\[ f_0(g) := f(g\gamma_0), \quad A := \tilde{\pi}(\gamma_0) \]
	so that $f_0 \in C^\infty_c(G(\R))$, then
	\[ \tilde{\pi}(f) = \pi(f_0) \circ A. \]
	Note that $A$ is an intertwining operator from $\omega \otimes \pi$ to $\pi \circ \Ad(\gamma_0)$. This will allow us to define the $\omega$-twisted character of $\tilde{\pi}$ as the distribution
	\[ \Theta_{\tilde{\pi}}: f \mapsto \Tr\left( \tilde{\pi}(f) \right) = \Tr\left( \pi(f_0) \circ A : V_\pi \to V_\pi \right). \]
	To be precise, one has to embed $V_\pi$ into a Hilbert globalizations of the associated Harish-Chandra module in order to talk about the trace; see \cite[\S 5.1]{BK14}.

	The distribution $\Theta_{\tilde{\pi}}$ is $K$-admissible. Indeed, it satisfies the equivariance \eqref{eqn:omega-character} under $G(\R) \simeq K(\R)$, and is clearly $\mathcal{Z}(\mathfrak{g}^{\mathrm{op}} \times \mathfrak{g})$-finite. When $\gamma_0$ can be chosen with $\Ad(\gamma_0) = \identity$, we revert to the Harish-Chandra characters.
\end{example}

\begin{example}[Relative characters]\label{eg:relative-characters}
	For $i = 1, 2$, let $H_i \subset G$ be a spherical subgroup and $\chi_i: \mathfrak{h}_i \to \CC$ be a reductive character. Let $\check{\pi}$ be the contragredient of an SAF representation $\pi$ of $G(\R)$. Consider continuous linear functionals that are equivariant under the Lie algebras $\mathfrak{h}_1$ and $\mathfrak{h}_2$:
	\[ \phi_1 \in \Hom_{\mathfrak{h}_1}(V_\pi, \chi_1), \quad \phi_2 \in \Hom_{\mathfrak{h}_2}(V_{\check{\pi}}, -\chi_2). \]

	Noting that $\pi^{\vee\vee} = \pi$, the corresponding \emph{relative character} is the distribution on $G(\R)$ (cf.\ \S\ref{sec:intro})
	\[ \Theta_{\phi_1, \phi_2}: f \mapsto \lrangle{\phi_1, \underbracket{\pi(f) \phi_2}_{\in V_\pi} }. \]
	These distributions are studied thoroughly in \cite{AGM16}, and the holonomicity has been established there; in \textit{loc.\ cit.}, $\Theta_{\phi_1, \phi_2}(f)$ is extended to all Schwartz functions $f$ on $G(\R)$.
	
	The conditions on $\phi_1, \phi_2$ and $\chi_1, \chi_2$ imply that $\Theta_{\phi_1, \phi_2}$ is an $H_1^{\mathrm{op}} \times H_2$-admissible distribution on $G(\R)$; in fact $D_G \cdot \Theta_{\phi_1, \phi_2}$ is an $H_1^{\mathrm{op}} \times H_2$-admissible $D_G$-module by Example \ref{eg:K-admissible} (ii).
	
	If the reductivity assumption on $\chi_1, \chi_2$ is dropped, $D_G \cdot \Theta_{\phi_1, \phi_2}$ is only $\mathfrak{h}_1^{\mathrm{op}} \times \mathfrak{h}_2$-admissible. Suppose for instance that $\chi_2$ is non-reductive, then $\Theta := \Theta_{\phi_1, \phi_2}$ is an irregular holonomic $D_G$-module. To see this, note that there exists an $H_1^{\mathrm{op}} \times H_2$-invariant open dense $U \subset G$ on which $\Theta$ is analytic (see below). There exists a copy of $\Ga$ in $R_u(H_2)$ on which $\chi_2$ is nontrivial. Were $D_G \cdot \Theta$ regular holonomic, so would be its pullback to any $\Ga$-orbit in $U$. However, $\Theta$ restricts to an exponential function on $\Ga$, whose $D$-module is irregular at $\infty$.
\end{example}

For the next result, we return to general $G$, $K$ and $Z$.

\begin{theorem}\label{prop:growth-estimate-admissible}
	Assume $Z$ is spherical homogeneous and $K \subset G$ is a spherical subgroup. Let $u$ be a $\mathfrak{k}$-admissible distribution on $Z(\R)$. There is a $K$-invariant open dense subset $U \subset Z$ such that $u$ is analytic on $U(\R)$. When $u$ is $K$-admissible, $P \cdot u|_{U(\R)}$ has moderate growth at infinity in the sense of Definition--Proposition \ref{def:moderate}, for all $P \in D_U$.
\end{theorem}
\begin{proof}
	Take a $\mathfrak{k}$-admissible module $\tilde{M}$ that contains $u$ in a subquotient. Corollary \ref{prop:holonomic} implies the existence of $U$.

	When $u$ is $K$-admissible, Theorem \ref{prop:admissible-regular} implies $\tilde{\mathcal{M}}$ is regular; so is its restriction to $U$, hence $\mathcal{M}_P := \mathscr{D}_U \cdot P u|_{U(\R)}$ are also regular holonomic, for any $P \in D_U$. Apply Theorem \ref{prop:growth-estimate} to $\mathcal{M}_P$, $\mu := P u|_{U(\R)}$, its analytic solution $P u|_{U(\R)}$ and to each connected component $V$ of $U(\R)$ to deduce the moderate growth at infinity.
\end{proof}

\begin{example}
	Consider the twisted character $\Theta_{\tilde{\pi}}$ (Example \ref{eg:HC-character}) for instance. As seen in Example \ref{eg:group-case-U} (with the notations therein), one can take $U := \tilde{G}_{\mathrm{reg}}$ inside $\tilde{G}$. Let us re-define $\Theta_{\tilde{\pi}}$ to be zero on $(\tilde{G} \smallsetminus U)(\R)$. We claim that Theorem \ref{prop:growth-estimate-admissible} implies that $\left| D^{\tilde{G}} \right|^a \Theta_{\tilde{\pi}}$ is locally bounded on $\tilde{G}(\R)$ for some $a > 0$. To see this, start with any smooth compactification $\tilde{G} \hookrightarrow \mathcal{G}$ and any $p: \mathcal{G}(\R) \to \R_{\geq 0}$ adapted to $U \hookrightarrow \mathcal{G}$ as in Definition--Proposition \ref{def:moderate}. For each $g \in G(\R)$, take a compact subanalytic neighborhood $E \ni g$ inside $\tilde{G}(\R)$. Since $E \cap (\mathcal{G} \smallsetminus U)(\R) = E \cap (\tilde{G} \smallsetminus U)(\R)$, we have $p(\gamma) = 0 \iff |D^{\tilde{G}}(\gamma)| = 0$ for all $\gamma \in E$. Theorem \ref{prop:inequality} implies that $p$ and $|D^{\tilde{G}}|$ are power-equivalent over $E$. Therefore $\Theta_{\tilde{\pi}} |_E$ is of $|D^{\tilde{G}}|$-bounded growth. This is considerably weaker than Harish-Chandra's result \cite[Theorem 3]{HC65} which attains $a = \frac{1}{2}$. The same estimates works for $P \cdot \Theta_{\tilde{\pi}}$ for any $P \in D_U$.
\end{example}

Let $\theta$ be a Cartan involution of $G$. When $K = G^\theta$ so that $K(\R) \subset G(\R)$ is a maximal compact subgroup. The classical technique of elliptic regularity applies. We rephrase it in the language of $\mathscr{D}$-modules as follows. It does not require sphericity of $Z$.

\begin{proposition}\label{prop:elliptic-regularity}
	Let $M$ be a $\mathfrak{k}$-admissible $D_Z$-module, then $\mathcal{M}^{\ana}$ is elliptic in the sense of \cite[Definition 11.5.5]{KS90}. The same is true for all $\mathscr{D}_Z$-submodules $\mathcal{N}$ of $\mathcal{M}$ Consequently,
	\[ \iHom_{\mathscr{D}_{Z^{\ana}}} \left(\mathcal{N}^{\ana}, \mathscr{A}_{Z(\R)}\right) \to \iHom_{\mathscr{D}_{Z^{\ana}}} \left(\mathcal{N}^{\ana}, \mathscr{B}_{Z(\R)}\right) \]
	is a quasi-isomorphism. Consequently, $\mathfrak{k}$-admissible hyperfunctions on $Z(\R)$ are analytic.
\end{proposition}
\begin{proof}
	To show the ellipticity of $\mathcal{M}$, we have to show that
	\[ T^*_{Z(\R)} Z(\CC) \cap \mathrm{Ch}(\mathcal{M}) = T^*_{Z(\R)} Z(\R). \]
	Here $Z(\CC)$ is viewed as a real manifold, $T^* Z(\CC)$ denotes the real cotangent bundle of $Z(\CC)$, containing the conormal bundle $T^*_{Z(\R)} Z(\CC)$ to $Z(\R)$. We have $T^*(Z^{\ana}) \simeq T^* Z(\CC)$ as real analytic manifolds by forgetting complex structures. Hence the intersection above makes sense.
	
	Let $\mathfrak{g} = \mathfrak{k} \oplus \mathfrak{p}$ be the decomposition into $\pm 1$-eigenspaces of $\theta$, where all vector spaces are over $\R$. There exists a non-degenerate bilinear form $\beta: \mathfrak{g} \times \mathfrak{g} \to \R$ which is negative definite (resp.\ positive definite) on $\mathfrak{k}$ (resp.\ on $\mathfrak{p}$). Let $\Omega \in \mathcal{Z}(\mathfrak{g})$ and $\Omega_K \in \mathcal{Z}(\mathfrak{k})$ be the Casimir elements corresponding to $\beta$ and $\beta|_{\mathfrak{k}}$. Let $\Delta := \Omega - 2\Omega_K$. It is homogeneous of degree $2$ in $U(\mathfrak{g})$, and Proposition \ref{prop:Ginzburg-bound} (or its proof) gives $\mathrm{Ch}(\mathcal{M}) \subset \bm{\mu}^{-1}(\Delta = 0)$.
	
	Take a basis $X_1, \ldots, X_a$ of $\mathfrak{k}$ and $X_{a+1}, \ldots, X_b$ of $\mathfrak{p}$ under which $\beta$ becomes
	\[ \text{diag}(\underbracket{-1, \ldots, -1}_{a\;\text{terms}}, \underbracket{1, \ldots, 1}_{b-a\;\text{terms}}). \]
	Then $\Delta = \sum_{i=1}^b X_i^2$ is negative definite on $\sqrt{-1} \cdot \mathfrak{g}^* \subset \mathfrak{g} \dotimes{\R} \CC$. Therefore the image of $T^*_{Z(\R)} Z(\CC)$ under $\bm{\mu}$ lies in $(\sqrt{-1} \cdot \mathfrak{g}^*) \cap \{ \Delta = 0 \} = \{0\}$. Upon recalling the definition of $\bm{\mu}$, we deduce $T^*_{Z(\R)} Z(\CC) \cap \mathrm{Ch}(\mathcal{M}) = T^*_{Z(\R)} Z(\R)$.
	
	For any subquotient $\mathcal{N}$ of $\mathcal{M}$, we have $\mathrm{Ch}(\mathcal{N}) \subset \mathrm{Ch}(\mathcal{M})$ thus $\mathcal{N}$ is elliptic as well. The quasi-isomorphism for solution complexes for $\mathcal{N}$ follows from \cite[p.468]{KS90}. By taking $\mathcal{N}$ to be the module generated by a $\mathfrak{k}$-admissible hyperfunction $u$ on $Z(\R)$, we infer that $u$ is analytic.
\end{proof}

\section{The case of generalized matrix coefficients}\label{sec:coeff}
The conventions from \S\ref{sec:app} remain in force. Moreover, we assume:
\begin{compactitem}
	\item $G$ is a connected reductive $\R$-group,
	\item $Z$ is a spherical homogeneous $G$-variety with $Z(\R) \neq \emptyset$,
	\item $K = G^\theta$ for some Cartan involution $\theta$ of $G$.
\end{compactitem}

We may choose $x_0 \in Z(\R)$ to identify $Z \simeq H \backslash G$ where $H$ is a spherical subgroup. As symmetric subgroups are spherical \cite[Theorem 26.14]{Ti11}, $K$-admissible distributions or hyperfunctions on $Z(\R)$ generate regular holonomic $\mathscr{D}_Z$-modules (Definition \ref{def:admissible-distribution}).

Next, we fix an SAF representation $\pi$ of $G(\R)$ (see \cite{BK14}). Set
\[ \EuScript{N}_\pi := \Hom_{G(\R)}\left(\pi, C^\infty(Z(\R))\right) \]
where we take the continuous $\Hom$ of continuous $G(\R)$-representations, and $C^\infty(Z(\R))$ is topologized as in \cite[\S 4.1]{Li18}. Let $V_\pi^{K\text{-fini}}$ denote the Harish-Chandra module of $K$-finite vectors in $V_\pi$.

It is well-known that $\dim_{\CC} \EuScript{N}_\pi$ is finite as $Z$ is spherical, see \cite[Theorem E]{AGM16} and the references therein, where stronger versions are obtained. We are going to show that the finiteness is an outright consequence of regularity, thereby giving a somewhat more geometric proof of this result. First, recall the localization functor $\Loc_Z$ from Example \ref{eg:localization}.

\begin{lemma}\label{prop:formal-fcn}
	Write $\widehat{\mathscr{O}}_{Z, x_0}$ for the formal completion of $\mathscr{O}_{Z, x_0}$. For any $U(\mathfrak{g})$-module $V$, we have an isomorphism of $\CC$-vector spaces
	\begin{align*}
		\Hom_{\mathscr{D}_{Z^{\ana}, x_0}}\left( \Loc_Z(V)^{\ana}, \widehat{\mathscr{O}}_{Z, x_0} \right) & \rightiso \Hom_{\CC}\left( V/\mathfrak{h}V, \CC \right) \\
		\Phi & \mapsto \left[ v + \mathfrak{h}V \mapsto \Phi(1 \otimes v)(x_0) \right]
	\end{align*}
	where $\Phi(1 \otimes v)(x_0)$ means the evaluation at $x_0$ of the formal function $\Phi(1 \otimes v)$.
\end{lemma}
\begin{proof}
	An element of the left hand side is the same as a $U(\mathfrak{g})$-homomorphism $V \to \widehat{\mathscr{O}}_{Z, x_0}$. Realize $\widehat{\mathscr{O}}_{G, 1}$ as the $\CC$-algebra of linear functions $U(\mathfrak{g}) \to \CC$. Observing that $Z \simeq H \backslash G$, we may identify $\widehat{\mathscr{O}}_{Z, x_0}$ with the $\CC$-subalgebra $\mathcal{F}$ of $\widehat{\mathscr{O}}_{G, 1}$ consisting of linear functions which are zero on $U(\mathfrak{h})U(\mathfrak{g})$.
	
	Note that the left $U(\mathfrak{g})$-action on $\widehat{\mathscr{O}}_{Z, x_0}$ transcribes to $(X f)(Y + U(\mathfrak{h})U(\mathfrak{g})) = f(YX + U(\mathfrak{h})U(\mathfrak{g}))$, where $X, Y \in U(\mathfrak{g})$ and $f \in \mathcal{F}$. Define:
	\[\begin{tikzcd}[row sep=small]
	\Hom_{U(\mathfrak{g})}(V, \mathcal{F}) \arrow[r, yshift=0.3em] & \Hom_{\CC}\left( V/\mathfrak{h}V, \CC \right) \arrow[l, yshift=-0.3em] \\
	\Phi \arrow[mapsto, r] & \Phi(\cdot)\left( 1 + U(\mathfrak{h}) U(\mathfrak{g}) \right) \\
	{[v \mapsto \Psi((\cdot)v) \in \mathcal{F} ]} & \Psi \arrow[mapsto, l]
	\end{tikzcd}\]
	It is routine to check that both arrows are well-defined, $\CC$-linear and and mutually inverse. The assertion follows.
\end{proof}

\begin{proposition}\label{prop:Hom-vs-Loc}
	Let $V$ be a Harish-Chandra module of $G$ and $K$. Then $\Hom_{\CC}\left( V / \mathfrak{h}V, \CC \right)$ is finite-dimensional. In fact it is isomorphic to $\Hom_{\mathscr{D}_{Z^{\ana}, x_0}}\left( \Loc_Z(V)^{\ana}, \mathscr{O}_{Z^{\ana}, x_0} \right)$.
\end{proposition}
\begin{proof}
	Let $\mathcal{M} := \Loc_Z(V)$. Example \ref{eg:K-admissible} (iii) together Corollary \ref{prop:admissible-regular} imply that $\mathcal{M}$ is regular holonomic. There is a natural homomorphism $\nu_{x_0}$ from the analytic local ring $\mathscr{O}_{Z^{\ana}, x_0}$ to $\widehat{\mathscr{O}}_{Z, x_0}$. The comparison theorem \cite[Proposition 7.3.1]{HTT08} implies that $\nu_{x_0}$ induces
	\[ \Hom_{\mathscr{D}_{Z^{\ana}, x_0}}\left( \mathcal{M}^{\ana}, \mathscr{O}_{Z^{\ana}, x_0} \right) \simeq \Hom_{\mathscr{D}_{Z^{\ana}, x_0}} \left( \mathcal{M}^{\ana}, \widehat{\mathscr{O}}_{Z, x_0} \right); \]
	the left hand side is finite-dimensional over $\CC$ by Kashiwara's constructibility theorem \cite[\S 4.6]{HTT08}, whilst the right hand side is $\Hom_{\CC}(V/\mathfrak{h}V, \CC)$ by Lemma \ref{prop:formal-fcn}.
\end{proof}

Note that the \emph{automatic continuity} for $\mathfrak{h} \subset \mathfrak{g}$ discussed in \cite[\S 11.2]{BK14} is still unreachable by these results.

\begin{corollary}
	For every SAF representation $\pi$ of $G(\R)$ we have $\dim_{\CC} \EuScript{N}_\pi < +\infty$, and $\mathfrak{h}V_\pi$ is of finite codimension in $V_\pi$.
\end{corollary}
\begin{proof}
	Put $V := V_\pi^{K\text{-fini}}$. We have $\Hom_{\CC}(V/\mathfrak{h}V, \CC) \subset \Hom_{H(\R)}(\pi, \CC)$ where the right hand side indicates the continuous $\Hom$. By Frobenius reciprocity in this setting, $\Hom_{H(\R)}(\pi , \CC) \simeq \Hom_{G(\R)}(\pi, C^\infty(H(\R) \backslash G(\R)))$. Write $Z(\R) = \bigsqcup_{i=1}^r x_i G(\R)$ with $H_i := \Stab_G(x_i)$, then
	\[ C^\infty(Z(\R)) = \bigoplus_{i=1}^r C^\infty(H_i(\R) \backslash G(\R)). \]
	
	Lemma \ref{prop:formal-fcn} applied to $H = H_1, \ldots, H_r$ gives $\dim \EuScript{N}_\pi < +\infty$.
\end{proof}

Let us turn to the functions $\eta(v) \in C^\infty(Z(\R))$ for $\eta \in \EuScript{N}_\pi$ and $v \in V_\pi$. They are called the \emph{generalized matrix coefficients} of $\pi$.

% Regularity can be combined with Theorem \ref{prop:growth-estimate-admissible} to estimate the growth of $\eta(v)$, although similar results have also been obtained by other methods, such as in \cite{KKS17}.

\begin{proposition}\label{prop:gen-coeff-admissible}
	For every $v \in V_\pi^{K\mathrm{-fini}}$ and $\eta \in \EuScript{N}_\pi$, the function $\eta(v)$ on $Z(\R)$ is $K$-admissible.
\end{proposition}
\begin{proof}
	This can be seen in two ways. Either apply Example \ref{eg:K-admissible} (i) together with Remark \ref{rem:unitarian} to see that $\eta(v)$ generates a $K$-admissible $D_Z$-module, or apply (iii) to see that $\mathscr{D}_Z \dotimes{U(\mathfrak{g})} V_\pi^{K\text{-fini}}$ is a $K$-admissible $\mathscr{D}_Z$-module and note that each $\eta \in \EuScript{N}_\pi$ induces a well-defined homomorphism
	\[ \mathscr{D}_Z \dotimes{U(\mathfrak{g})} V_\pi^{K\text{-fini}} \to \mathscr{C}^\infty_{Z(\R)}, \quad P \otimes v \mapsto P \eta(v) \]
	between $\mathscr{D}_Z$-modules.
\end{proof}

By combining Theorem \ref{prop:growth-estimate} and Proposition \ref{prop:gen-coeff-admissible}, it is possible to deduce the estimate below for generalized matrix coefficients. Recall from \cite[p.51]{BK14} that a continuous semi-norm $q: V_\pi \to \R_{\geq 0}$ is called \emph{$G$-continuous} if $G(\R) \times V_\pi \to V_\pi$ is continuous with respect to the $q$-topology on $V_\pi$.

\begin{theorem}\label{prop:coeff-estimate}
	Let $\eta \in \EuScript{N}_\pi$. There exist
	\begin{compactitem}
		\item a function $w: Z(\R) \to \R_{\geq 1}$ as in Lemma \ref{prop:weight},
		\item a $G$-continuous semi-norm $q: V_\pi \to \R_{\geq 0}$,
	\end{compactitem}
	such that $|\eta(v)(x)| \leq w(x) q(v)$ for all $v \in V_\pi$ and $x \in Z(\R)$. They depend on $\pi$ and $\eta$.
\end{theorem}

Using the theory of toroidal embeddings \cite[\S 7]{KK16} together with Łojasiewicz's inequality, this can be upgraded to an estimate in terms of the weak polar decomposition \cite[\S 13]{KK16}. Some further definitions are in order.

Fix a minimal parabolic $\R$-subgroup $P \subset G$, a Levi component $M_P$ of $P$, and let $A$ be the maximal split central torus in $M_P$. Using the local structure theorems for $Z$ (see \cite[4.6 Corollary and (4.15)]{KK16}), one attaches to $Z$ an affine smooth subvariety $Z_{\mathrm{el}} \subset Z$ (the \emph{elementary kernel}), $Z_{\mathrm{el}}(\R) \neq \emptyset$, on which $A$ acts with kernel $A_0$. Set $A_Z := A/A_0$ which acts freely on $Z_{\mathrm{el}}$. We may take $x_0 \in Z_{\mathrm{ell}}(\R)$.

In \cite[(10.9)]{KK16} is defined the set of simple restricted spherical roots $\Sigma_{\R}(Z) \subset X^*(A)$. Following \cite[(13.1), (13.5)]{KK16} we define
\begin{align*}
	A_Z(\R)^- & := \left\{ a \in A_Z(\R) : \forall \sigma \in \Sigma_{\R}(Z), \; |\sigma(a)| \leq 1 \right\}, \\
	\mathfrak{a}_Z^* & := X^*(A_Z) \otimes \R.
%	\mathfrak{a}_Z & := X_*(A_Z) \otimes \R, \\
%	\mathfrak{a}_Z^- & := \left\{ \mu \in \mathfrak{a}_Z: \forall \sigma \in \Sigma_{\R}(Z), \; \lrangle{\sigma, \mu} < 0 \right\}.
\end{align*}

For all $\lambda = \sum_i \lambda_i \otimes t_i \in \mathfrak{a}_Z^*$ and $a \in A_Z(\R)$, we write $|a|^\lambda := \prod_i |\lambda_i(a)|^{t_i}$, which is well-defined.

% The maximal compact subgroup of $A_Z(\R)$ will be denoted by $\mathcal{K}$; it is a finite $2$-group on which $|\cdot|^\lambda = 1$, and $A_Z(\R) \simeq \mathcal{K} \times \mathfrak{a}_Z$.

Also needed is the affine group of $\R$-central automorphisms $\mathfrak{A} := \mathfrak{A}_{\R}(Z)$ acting on the \emph{left} of $Z$; see \cite[(8.5)]{KK16}. Its identity connected component $\mathfrak{A}^\circ$ is a split torus embedded in $A_Z$. Given an SAF representation $\pi$, note that $\mathfrak{A}(\R)$ acts linearly and continuously on $\EuScript{N}_\pi$ by
\[ (a \eta)(v)(x) = \eta(v)(a^{-1}x), \quad a \in \mathfrak{A}(\R), \; v \in V_\pi, \; x \in Z(\R). \]
The eigen-embeddings in $\EuScript{N}_\pi$ are defined to be the eigenvectors under $\mathfrak{A}^\circ(\R)$.

\begin{corollary}\label{prop:polar-estimate}
	Let $\pi$ be an SAF representation and $\eta \in \EuScript{N}_\pi$ be an eigen-embedding. For any closed subanalytic subset $\Omega \subset G(\R)$, there exist $\lambda \in \mathfrak{a}_Z^*$ and a $G$-continuous semi-norm $q: V_\pi \to \R_{\geq 0}$, both depending on $(\pi, \eta)$, such that
	\[ \left| \eta(v)(x_0 a\omega) \right| \leq |a|^{\lambda} q(v), \quad a \in A_Z(\R)^-, \; \omega \in \Omega. \]
\end{corollary}

This is comparable to \cite[Theorem 7.2]{KKS17}. However, the crude estimate above can be deduced more directly from the notion of SAF representations: they are of \emph{moderate growth}. In \textit{loc.\ cit.}, one obtains the optimal exponent $\lambda$, and the approach thereof can be extended to all real spherical homogeneous spaces; see \cite{BDKS17}. For this reason, the proofs of both Theorem \ref{prop:coeff-estimate} and Corollary \ref{prop:polar-estimate} are omitted here.

% The following is for bibtex - preferred on arXiv
\bibliographystyle{abbrv}
\bibliography{Regularity}

% Below: for Biblatex... Comment out for arXiv
% \printbibliography[heading=bibintoc]

\vspace{1em}
\begin{flushleft} \small
	Beijing International Center for Mathematical Research / School of Mathematical Sciences, Peking University. No.\ 5 Yiheyuan Road, Beijing 100871, People's Republic of China. \\
	E-mail address: \href{mailto:wwli@bicmr.pku.edu.cn}{\texttt{wwli@bicmr.pku.edu.cn}}
\end{flushleft}

\end{document}